\documentclass[11pt,a4paper]{amsart}
\usepackage{amssymb,amsmath,amscd,url,geometry,pdflscape}

\usepackage{graphicx, young, multirow, diagbox}

\newtheorem{theorem}{Theorem}[section] 
\newtheorem{lemma}[theorem]{Lemma}     
\newtheorem{corollary}[theorem]{Corollary}

\theoremstyle{remark}
\newtheorem{remark}[theorem]{Remark}
\newtheorem{example}[theorem]{Example}

\theoremstyle{definition}
\newtheorem{definition}[theorem]{Definition}

\usepackage{amsfonts,url}

\newcommand{\q }{/\!/}

\newcommand{\Q}{\mathbb{Q}}
\newcommand{\R}{\mathbb{R}}

\newcommand{\Z}{\mathbb{Z}}
\newcommand{\C}{\mathbb{C}}

\def\SL{{\rm SL}}

\def\bU{{S^1}}

\def\SL{{\rm SL}}
\def\PSL{{\rm PSL}}

\def\PSU{{\rm PSU}}

\def\SU{{\rm SU}}
\def\SL{{\rm SL}}

\def\PGL{{\rm PGL}}

\def\cG{{\mathcal G}}

\def\A{{\mathbb A}}

\def\bS{{\mathbf S}}
\def\bG{{\mathbf G}}

\def\fS{{\mathfrak S}}

\def\q {{/\!/}}

\def\rhohat{\widehat{\rho_1}}
\def\rhohatprime{\widehat{\rho'_1}}
\def\zetilde{\widetilde{\zeta}}
\def\yes{\text{Yes}}
\def\no{\text{No}}

\title{Stratified Langlands duality in the $A_n$ tower}
\author{Graham A. Niblo, Roger Plymen and Nick Wright}
\email{g.a.niblo@soton.ac.uk, r.j.plymen@soton.ac.uk, n.j.wright@soton.ac.uk} 
\address{Mathematical Sciences, Southampton University, SO17 1BJ, UK}
\begin{document}
\begin{abstract} Let $\bS_k$ denote a maximal torus in the complex Lie group $\bG = \SL_n(\C)/C_k$ and let $T_k$ denote a maximal torus in its compact real form $\SU_n(\C)/C_k$,  where $k$ divides $n$. 
Let $W$ denote the Weyl group of $\bG$, namely the symmetric group $\fS_n$.   
We elucidate the structure of the extended quotient $\bS_k \q W$ as an algebraic variety and of $T_k\q W$ as a topological space, in both cases describing them as bundles over unions of tori.  
Corresponding to the invariance of K-theory under Langlands duality, this calculation provides a homotopy equivalence between $T_k\q W$ and its dual $T_{\frac{n}{k}}\q W$. Hence there is an isomorphism in cohomology for the extended quotients which is stratified as a direct sum over conjugacy classes of the Weyl group. We use our formula to compute a number of examples.

\end{abstract}
\maketitle
\tableofcontents
\section{Introduction}

In \cite{NPW} we gave a new proof of the Baum-Connes conjecture for the extended affine Weyl groups associated to compact connected semi-simple Lie groups, and showed that Langlands duality for the Lie groups induces an isomorphism of $K$-theory for the corresponding extended affine Weyl groups. In his consideration of the K-theory of Hecke algebras, Solleveld \cite{S1, S2} determined the right hand side of the Baum-Connes conjecture for the affine Weyl groups associated to a number of Lie Groups including $\SL_n(\C)$.  In order to do so he computed the extended quotients $T_n\q W$ up to homotopy by constructing the quotients $\bS_n\q W$.  In this paper we consider the case of all semisimple (real and complex) Lie groups of type $A_n$, and give a detailed construction of the extended quotients $\bS_k \q W$ and $T_k\q W$. In particular we show that there is a homotopy equivalence between $T_k\q W$ and $T_{\frac{n}{k}}\q W$, inducing isomorphisms in cohomology generalising the isomorphism at the level of $K$-theory given in \cite{NPW}. Furthermore our construction shows that the homotopy equivalence is stratified by conjugacy classes of the Weyl group. 

In the complex case, the extended quotients are varieties which are unions of tori crossed with cyclic singularities, and we compute the structure of these varieties. We illustrate briefly the connection with $p$-adic groups.   Let $F$ be a local non-archimedean field such as the $p$-adic field $\Q_p$, and consider the $p$-adic group  
\[
\cG = \PGL_n(F).
\]   
The Langlands dual of $\cG$ is $\bG = \SL_n(\C)$.   Suppose further that $p > n$.   Let $\bS$ be the standard maximal torus in $\SL_n(\C)$, let $T$ be the maximal compact subgroup of $\bS$.      According to \cite[Theorem 5.2]{ABPS}, the Iwahori-spherical block in the principal series of $\cG$ admits the structure of the algebraic variety $\bS\q W$.   Furthermore,  the set of tempered representations in the Iwahori-spherical block admits the structure of the extended quotient  $T\q W$. Theorem \ref{maintheorem}  and Theorem \ref{maintheorem2}, with $k = 1$ therefore provide  explicit descriptions of these spaces of representations.   

\bigskip

We now proceed to the statement of our results. Let $C_k$ denote the cyclic group of $k$-th roots of unity in $\C$. If $k$ divides $n$ then, abusing notation, we will identify this group with the subgroup $C_k I_n$ in the centre of $\SL_n(\C)$. We will first consider the Lie group $\SL_n(\C)/C_k$ where the corresponding Weyl group is the symmetric group ${\fS}_n$. Its conjugacy classes, indexed by cycle structures, correspond to partitions of $n$.

\begin{definition} For a partition $\mu$ of $n$ given by $n=\mu_1+\ldots + \mu_c$ the numbers $\mu_1\ldots ,\mu_c$ will be called the \emph{parts} of $\mu$ and the parts have \emph{multiplicities}  $m_j=|\{i\mid \mu_i=j\}|$. We will denote the greatest common divisor of the parts of $\mu$ by $g(\mu)$ and the greatest common divisor of the multiplicities of the parts by $m(\mu)$. We also let $b(\mu)$ denote the number of distinct parts of $\mu$ and will write $c(\mu)$ for the total number of parts. For each $i$ let $p_i(\mu)$ denote the number of distinct parts of $\mu$ which have multiplicity strictly greater than $i$. We will omit the $\mu$ decoration on these symbols where the context is clear.
\end{definition}

\begin{theorem}\label{maintheorem}
Let $\bS_k$ denote a maximal torus of $\SL_n(\C)/C_k$ and $W$ denote the corresponding Weyl group. The extended quotient $\bS_k\q W$ is an algebraic variety which decomposes as a disjoint union of irreducible components. Each component is the product of a complex torus with a cyclic quotient of a complex affine space. Specifically we have a disjoint union over partitions $\mu$ of $n$:

\[
\bS_k\q W\cong \coprod_\mu\coprod\limits_{\omega\in C_{\gcd(g(\mu), k)}}\  (\C^\times)^{b(\mu)-1}\times \A_{\mu,\omega}\times X_{\mu,\omega}.
\]
\noindent The space $X_{\mu,\omega}$ is a discrete set of cardinality $\gcd(g(\mu)/|\omega|,n/k)$. The space $\A_{\mu,\omega}$ is the cyclic singularity 
\[
\A^{c(\mu)-b(\mu)}/C_{\gcd(m(\mu),\frac{k}{|\omega|})}
\]
where the generator $\eta$ of the group $C_{\gcd(m(\mu),\frac{k}{|\omega|})}$ acts by multiplication by powers of $\eta$ on the factors of $\A^{c(\mu)-b(\mu)}$. For each $i$ there are $p_i(\mu)$ factors on which the generator acts by multiplication by $\eta^i$.
\end{theorem}

\begin{corollary}\label{corollary}
With $\bS_k$ and $W$ as above there is a homotopy equivalence

\[
\bS_k\q W\cong \coprod_\mu\  (\C^\times)^{b(\mu)-1}\times Y_{\mu}.
\]
where $Y_\mu$ is a discrete set of cardinality   $\displaystyle \frac{g(\mu)}a\sum_{s=0}^{a-1} \gcd(a,s)$ where $\displaystyle a=\gcd\left(g(\mu),\frac n {g(\mu)}, k,\frac n k\right)$.

In particular, interchanging $\SL_n(\C)/C_k$ with its Langlands dual does not change the homotopy type of the extended quotient.
\end{corollary}

\begin{remark}
While passing from  $\SL_n(\C)/C_k$ to the Langlands dual group $\SL_n(\C)/C_{n/k}$ does not change the homotopy type of the extended quotient, the singularity structure of the cyclic quotients can and often does vary. In particular when $k=1$ the varieties are smooth, while the dual case $k=n$, considered by Solleveld, will always give the most singular case in the tower. In that case each $Y_\mu$ has cardinality $g(\mu)$,  recovering Solleveld's formula, \cite{S1}. 
\end{remark}

\begin{remark}
The number $\sum_{s=0}^{a-1} \gcd(a,s)$ appearing in the formula for cardinality of the set $Y_\mu$ is the $a$th value of Pillai's arithmetical function \cite{P}. The cardinality of $|Y_\mu$ can alternatively be expressed as $g(\mu)$ times the sum $\sum\limits_{d|a}\frac{\phi(d)}{d}$.\end{remark}

We now consider the real case.

\begin{theorem}\label{maintheorem2}
Let $T_k$ denote a maximal torus of $\SU_n(\C)/C_k$ and $W$ denote the corresponding Weyl group. The extended quotient $T_k\q W$ is  a disjoint union of compact orbifolds with boundary. Each component is a bundle over a compact torus with fibre a cyclic quotient of a polysimplex. Specifically we have a disjoint union over partitions $\mu$ of $n$:

\[
T_k\q W\cong \coprod_\mu\coprod\limits_{\omega\in C_{\gcd(g(\mu), k)}}\  E_{\mu,\omega}/C_{\gcd(m(\mu),\frac{k}{|\omega|})} \times X_{\mu,\omega}.
\]
\noindent The space $X_{\mu,\omega}$ is a discrete set of cardinality $\gcd(g(\mu)/|\omega|,n/k)$. The space $E_{\mu,\omega}$ is a bundle of polysimplices over a torus of dimension $b(\mu)-1$.

The group $C_{\gcd(m(\mu),\frac{k}{|\omega|})}$ preserves the torus and acts on the polysimplicial fibres, which are products of simplices of dimensions $m_j-1$, where $j$ ranges over the distinct parts of the partition $\mu$.  Each simplex can be regarded as a join of $m_j/\gcd(m(\mu),\frac{k}{|\omega|})$ simplices of dimension $\gcd(m(\mu),\frac{k}{|\omega|})-1$ on each of which which the group acts by cyclically permuting the vertices.  This action preserves the orientation on the fibres if and only if $c$ is odd or the $2$-adic norms satisfy $|c|_2<|\gcd(m(\mu),\frac{k}{|\omega|})|_2$.
\end{theorem}

The precise construction of the bundle $E_{\mu,\omega}$ is given in section \ref{proof of main2} below.

\begin{corollary}\label{corollary2}
With $T_k$ and $W$ as above there is a homotopy equivalence:

\[
T_k\q W\cong \coprod_\mu\  (\bU)^{b(\mu)-1}\times Y_{\mu}.
\]
where $Y_\mu$ is a discrete set of cardinality   $\displaystyle \frac{g(\mu)}a\sum_{s=0}^{a-1} \gcd(a,s)$ where $\displaystyle a=\gcd\left(g(\mu),\frac n {g(\mu)}, k,\frac n k\right)$.

In particular, interchanging $\SU_n(\C)/C_k$ with its Langlands dual does not change the homotopy type of the extended quotient.
\end{corollary}

The equivariant Chern character for discrete groups \cite{BC} gives a map
\[
\mathrm{ch}_W : K^j_W(T_k) \to \oplus_l  H^{j + 2l} (T_k \q W; \C)
\]
which becomes an isomorphism when $K^j_W(T_k)$ is tensored with $\C$.
Corollary \ref{corollary2} may then be viewed as  a refinement of the results of \cite{NPW}, stratifying the $K$-theory isomorphism given there over the conjugacy classes of the Weyl group and additionally giving an isomorphism at the level of cohomology. Our formulas then allow the computation of cohomology and K-theory groups and we include as appendices a number of tables of these computations.

\section{Examples}\label{examples}

\begin{example}
We consider the examples of $\SL_6(\C)$ and $\PSL_6(\C)=\SL_6(\C)/C_6$. In the former case, since $k=1$, the variable $\omega$ in the summation formula can only take the value $1$, and we give the variety $(\C^\times)^{b(\mu)-1}\times \A_{\mu,1}$ and its multiplicity $|X_{\mu,1}|$ for each partition $\mu$ of $6$.  In the latter case each partition will give rise to the same number of components as for $\SL_6(\C)$, but in this case they are indexed by the value of $\omega$ which ranges over certain powers of the primitive $6$th root of unity $\zeta$.
We list the possible values of $\omega$ (each giving a single component) and again describe the corresponding variety $(\C^\times)^{b(\mu)-1}\times \A_{\mu,\omega}$.

\[
\begin{array}{c|c|c|c|c}
 & \multicolumn{1}{r}{} &\SL_6\hfill& \multicolumn{2}{c}{\PSL_6}\\\hline
{\mu^{\mathstrut}}_{\mathstrut} & |X_{\mu,1}| & (\C^\times)^{b(\mu)-1}\times \A_{\mu,1} & \omega & (\C^\times)^{b(\mu)-1}\times \A_{\mu,\omega}\\
\hline
\multirow{6}{*}{6} & \multirow{6}{*}{6} & \multirow{6}{*}{$\A^0$} &1& \A^0\\
 & &  &\zeta & \A^0\\
 & &  &\zeta^2 & \A^0\\
 & &  &-1 & \A^0\\
 & &  &\zeta^4 & \A^0\\
 & &  &\zeta^5 & \A^0\\\hline
1+5 & 1 & \C^\times &1& \C^\times\\\hline
\multirow{2}{*}{2+4} & \multirow{2}{*}{2} &  \multirow{2}{*}{$\C^\times$} &1& \C^\times\\
&&&-1& \C^\times\\\hline
1+1+4 & 1 & \C^\times\times\A^1 &1& \C^\times\times\A^1\\\hline
\multirow{3}{*}{3+3} & \multirow{3}{*}{3} & \multirow{3}{*}{$\A^1$} &1& \A^1/\langle-1\rangle\\
&&&\zeta^2& \A^1/\langle-1\rangle\\
&&&\zeta^4& \A^1/\langle-1\rangle\\\hline
1+2+3 & 1 & (\C^\times)^2 &1& (\C^\times)^2\\\hline
1+1+1+3 & 1 & \C^\times\times\A^2 &1& \C^\times\times\A^2\\\hline
\multirow{2}{*}{2+2+2} & \multirow{2}{*}{2} & \multirow{2}{*}{$\A^2$} &1&\A^2/\langle(\zeta^2,\zeta^4)\rangle\\
&&&-1&\A^2/\langle(\zeta^2,\zeta^4)\rangle\\\hline
1+1+2+2 & 1 & \C^\times\times\A^2 &1& \C^\times\times\A^2/\langle(-1,-1)\rangle\\\hline
1+1+1+1+2 & 1 & \C^\times\times\A^3 &1& \C^\times\times\A^3\\\hline
1+1+1+1+1+1 & 1 & \A^5 &1& \A^5/\langle(\zeta^1,\zeta^2,\zeta^3,\zeta^4,\zeta^5)\rangle\\\hline
\end{array} 
\]

Note that the values of $\omega$ allowed for $\PSL_6(\C)$ are exactly the set $|X_{\mu,1}|$ for $\SL_6(\C)$.  This is a general property of the cases $k=1$ and $k=n$, however for intermediate cases the picture is more complicated. For each $\mu$, the homotopy-types of the corresponding components agree, and in many cases the components are isomorphic as varieties.  As noted above this is not true in general and indeed the component varieties for $\SL_6$ and $\PSL_6$ are not isomorphic (or even homeomorphic) in the cases $\mu=2+2+2$, $\mu=1+1+2+2$ and $\mu=1+1+1+1+1+1$.  While for $\SL_6$, the varieties are smooth and the factor $\A_{\mu,1}$ is simply an affine space,  for $\PSL_6$ the factor $\A_{\mu,\omega}$ is a cyclic singularity as classified in \cite{F}.

We remark that in the $\PSL_6$ example the spaces $\A_{\mu,\omega}$ do not in fact depend on $\omega$.  This is because $6$ is square-free, hence $\gcd(m,\frac k{|\omega|})=\gcd(m,k)$ for all $\omega$.
\end{example}

\begin{example}
We now consider the dual examples of $\SL_6(\C)/C_2$ and $\SL_6(\C)/C_3$. Here we will see that, while the components again do not depend on the variable $\omega$ the possible values of $\omega$ associated to a partition $\mu$ do depend on $k$, as does the cardinality of $X_{\mu, \omega}$. Again we also note that the singularity structure is differs between the two groups for the partitions $2+2+2$, $1+1+2+2$ and $1+1+1+1+1+1$. 
\[
\begin{array}{c|c|c|c|c|c|c}
 & \multicolumn{2}{r}{} &\SL_6/C_2\hfill& \multicolumn{3}{c}{\SL_6/C_3}\\\hline

{\mu^{\mathstrut}}_{\mathstrut} & \omega &  |X_{\mu,\omega}| & (\C^\times)^{b(\mu)-1}\times \A_{\mu,\omega} & \omega & |X_{\mu,\omega}| & (\C^\times)^{b(\mu)-1}\times \A_{\mu,\omega}\\
\hline
\multirow{6}{*}{6} & \multirow{3}{*}{1} & \multirow{3}{*}{3} & \multirow{3}{*}{$\A^0$}& \multirow{2}{*}{1} &\multirow{2}{*}{2}& \multirow{2}{*}{$\A^0$}\\
 & & & & & & \\
 & & & & \multirow{2}{*}{$\zeta^2$} &\multirow{2}{*}{2}& \multirow{2}{*}{$\A^0$}\\
 & \multirow{3}{*}{-1}& \multirow{3}{*}{3} & \multirow{3}{*}{$\A^0$} & & &\\
 & &  &&\multirow{2}{*}{$\zeta^4$} &\multirow{2}{*}{2}& \multirow{2}{*}{$\A^0$}\\
 & & & & & & \\\hline
1+5 & 1 & 1 & \C^\times &1& 1 & \C^\times\\\hline
\multirow{2}{*}{2+4} & 1 & 1 &  \C^\times&\multirow{2}{*}{1}&\multirow{2}{*}{2}&\multirow{2}{*}{$\C^\times$}\\
&-1 & 1&\C^\times&&& \\\hline
1+1+4 & 1 &1 & \C^\times\times\A^1 &1& 1 &\C^\times\times\A^1\\\hline
\multirow{3}{*}{3+3} & \multirow{3}{*}{1} & \multirow{3}{*}{3} & \multirow{3}{*}{$\A^1/\langle-1\rangle$} &1&1& \A^1\\
&&&&\zeta^2& 1&\A^1\\
&&&&\zeta^4& 1&\A^1\\\hline
1+2+3 & 1 & 1 & (\C^\times)^2 &1& 1 & (\C^\times)^2\\\hline
1+1+1+3 & 1 & 1 & \C^\times\times\A^2 &1 &1 & \C^\times\times\A^2\\\hline
\multirow{2}{*}{2+2+2} & 1&1&\A^2&\multirow{2}{*}{1}&\multirow{2}{*}{2}&\multirow{2}{*}{$\A^2/\langle(\zeta^2,\zeta^4)\rangle$}\\
&-1&1&\A^2&&&\\\hline
1+1+2+2 & 1 &1 &  \C^\times\times\A^2/\langle(-1,-1)\rangle &1&1 &  \C^\times\times\A^2\\\hline
1+1+1+1+2 & 1& 1 & \C^\times\times\A^3 &1&1 & \C^\times\times\A^3\\\hline
1+1+1+1+1+1 & 1& 1 & \A^5/\langle(-1,1,-1,1,-1)\rangle &1&1 & \A^5/\langle(\zeta^2,\zeta^4,1,\zeta^2,\zeta^4)\rangle\\\hline
\end{array} 
\]
\end{example}

\bigskip
To illustrate the possible dependence of $X_{\mu,\omega}$ and $\A_{\mu,\omega}$ on $\omega$ we must consider a value of $n$ with square factors.  Specifically we will consider $n=16$.  Since there are rather a lot of partitions of $16$ we shall just select a few examples to demonstrate the process by which the components of $\bS_k\q W$ are constructed.

\begin{example}
[The quotient variety for $SL_{16}/C_2$ corresponding to $\mu=2+2+2+2+4+4$] The partition has greatest common divisor $g=2$, so the total number of components is $|Y_\mu|=\displaystyle \frac{2}a\sum_{s=0}^{a-1} \gcd(a,s)= 3$ where $\displaystyle a=\gcd\left(2,\frac n {2}, k,\frac n k\right)=2$.

The greatest common divisor $\gcd(g,k)=2$ hence $\omega=\pm 1$. In the case $\omega=1$ we have $|X_{\mu,\omega}|=\gcd(\frac{g}{|\omega|},\frac{n}{k})=\gcd(2,8)=2$. The multiplicities are $m_2=4, m_4=2$ so $m=2$ so the cyclic singularity is the quotient of $\A^{6-2}$ by the cyclic group $C_{\gcd(m,\frac{k}{|\omega|})}=C_{\gcd(2,8)}=C_2$. This is generated by $\eta=-1$ which acts by multiplication by $(\eta^1,\eta^1, \eta^2, \eta^3)=(-1,-1,1,-1)$. The exponents of $\eta$ appearing here can easily be read as the non-zero entries of the Young tableau decorated as follows:
\medskip
\[\begin{Young}
$0$&&&\cr
$1$&&&\cr
$0$&\cr
$1$&\cr
$2$&\cr
$3$&\cr
\end{Young}
\]
The $b$ zeros appearing in the table correspond to (the trivial action on) the torus factor, which is a codimension-1 torus in $\C^{b}$.

Turning to the case $\omega=-1$ we see that $|X_{\mu,\omega}|=\gcd(\frac{g}{|\omega|},\frac{n}{k})=\gcd(1,8)=1$. The cyclic group $C_{\gcd(m,\frac{k}{|\omega|})}=C_{\gcd(2,1)}$ is trivial,  hence for $\omega=1$ we obtain a copy of $\C^\times \times \A^4$. In conclusion this partition yields three components, two copies of the space $\C^\times \times \A^4/(-1,-1,1,-1)$ indexed by $\omega=1$ and one copy of $\C^\times \times \A^4$ indexed by $\omega=-1$. \end{example}
\bigskip
We now consider the Langlands dual case $k=8$:
\begin{example} [The quotient variety for $SL_{16}/C_8$ corresponding to $\mu=2+2+2+2+4+4$]
\ 

\noindent Again $|Y_\mu|=3$, since exchanging $k=2$ with $k=8$ interchanges $k$ with $n/k$ and we obtain the same value $a=2$. As above $g=m=2$ and $\omega=\pm 1$. In the case $\omega=1$ we again have $|X_{\mu, \omega}|=2$ and we obtain two copies of the variety $\C^\times \times \A^4/(-1,-1,1,-1)$. When $\omega=-1$ we have $|X_{\mu,\omega}|=\gcd(1,2)=1$, and the cyclic group is $C_{\gcd(m,\frac{k}{|\omega|})}=C_{\gcd(2,4)}=C_2$. This gives the variety $\A^4/(-1,-1,1,-1)$ in contrast to the case $k=2$ where we obtained a copy of the affine space itself  . Hence this partition now yields three copies of the space $\C^\times \times \A^4/(-1,-1,1,-1)$, two indexed by $\omega=1$ and one indexed by $\omega=-1$.
\end{example}

We remark that for this particular choice of partition, taking $k=4$ (the self dual case) would yield the same quotient variety as we obtain for $k=8$. By way of contrast consider the components arising for the partition $\mu=4+4+4+4$.

\begin{example}[The quotient variety for $SL_{16}/C_4$ and $SL_{16}/C_8$ corresponding to the partition $\mu=4+4+4+4$]
We have $g=m=4$, and $\omega$ lies in $C_{\gcd(g,k)}=C_4$ hence $\omega\in\{\pm 1,\pm i\}$. With $\omega=1$ we have $|X_{\mu,\omega}|=\gcd(\frac{g}{|\omega|},\frac{n}{k})=\gcd(4,4)=4$. The cyclic singularity is the quotient of $\A^{4-1}$ by the cyclic group $C_{\gcd(4,4)}=C_4$. This is generated by $\eta=i$ which acts by multiplication by $(\eta^1,\eta^2, \eta^3)=(i,-1,-i)$:\medskip
\[\begin{Young}
$0$&&&\cr
$1$&&&\cr
$2$&&&\cr
$3$&&&\cr
\end{Young}
\]

Now set $\omega=-1$. We have $|X_{\mu,\omega}|=\gcd(2,4)=2$. The cyclic singularity is the quotient of $\A^{3}$ by the cyclic group $C_{\gcd(4,2)}=C_2$. This is generated by $-1$ which acts by multiplication by $(-1,1,-1)$.

For each of $\omega=\pm i$, $|X_{\mu,\omega}|=\gcd(1,4)=1$ and the cyclic group acting is trivial because $k/|\omega|=1$ so these two elements yield, between them, two copies of the affine space $\A^3$.

Hence this partition furnishes $8$ components, four of them isomorphic to the cyclic singularity $\A^3/(i,-1,-i)$, two of them isomorphic to $\A^3/(-1,1,-1)$ and two of them isomorphic to $\A^3$.

Now turning to the case  $k=8$ again we have, $\omega\in\{\pm 1,\pm i\}$ since $\gcd(4,8)=4$. 
With $\omega=1$ we have $|X_{\mu,\omega}|=\gcd(\frac{g}{|\omega|},\frac{n}{k})=\gcd(4,2)=2$. The cyclic singularity is the quotient of $\A^{3}$ by the cyclic group $C_{\gcd(4,8)}=C_4$, yielding $\A^3/(i,-1,-i)$.
For $\omega=-1$, $|X_{\mu,\omega}|=\gcd(2,2)=2$. The cyclic singularity is the quotient of $\A^{3}$ by the cyclic group $C_{\gcd(4,4)}=C_4$. This is generated by $i$ which acts by multiplication by $(i,-1,-i)$.

For each of $\omega=\pm i$, $|X_{\mu,\omega}|=\gcd(1,2)=1$ and the cyclic group acting is has order $\gcd(4,2)=2$. Hence these two elements yield, between them, two copies of the cyclic singularity $\A^3/(-1,1,-1)$.

Hence this partition now furnishes only $6$ components (rather than the $8$ appearing in the $k=4$ case). Four of these are isomorphic to the cyclic singularity $\A^3/(i,-1,-i)$ and two of them are isomorphic to $\A^3/(-1,1,-1)$. In general the number of components grows with $\gcd(k, n/k)$. However if $n$ is square free the number of components will be constant in $k$. See the examples in table \ref{K(n,k)}.

\end{example}

\section{The extended quotient}

Let $\Gamma$ be a finite group acting on  a complex affine variety $X$ 
by automorphisms,  
\[ 
\Gamma \times X \to X.
\]
The quotient variety $X/\Gamma$ is obtained by collapsing each orbit to a point. 

 For $x\in X $, $\Gamma_x$ denotes the stabilizer group of $x$:
$$ 
\Gamma_x = \{\gamma\in \Gamma : \gamma x = x\}.
$$
Let $c(\Gamma_x)$ denote the set of conjugacy classes of  $\Gamma_x$. The extended quotient 
is obtained by replacing the orbit of $x$ by $c(\Gamma_x)$. This is done in the following way.

 Set 
\[\widetilde{X}: = \{(\gamma, x) \in \Gamma \times X : \gamma x = x\}.
\] 
Then $\widetilde{X}$ is an affine variety and is a subvariety of $\Gamma \times X$. 
The group $\Gamma$ acts on $\widetilde{X}$:
\begin{align*}
\Gamma &\times \widetilde{X} \to \widetilde{X}\\
\alpha(\gamma, x) = & (\alpha\gamma \alpha^{-1}, \alpha x), \quad\quad \alpha \in \Gamma,
\quad (\gamma, x) \in \widetilde{X}.
\end{align*}

\noindent The (geometric) extended quotient  $ X/\!/\Gamma $ is, by definition, the usual quotient for the action of 
$\Gamma$ on $\widetilde{X}$:
\[
X \q \Gamma: = \widetilde{X} / \Gamma.
\]
The projection 
\[
\widetilde{X} \to X,  \quad (\gamma, x) \mapsto x
\]
 is $\Gamma$-equivariant 
and so passes to quotient spaces to give a morphism of affine varieties
\[
\rho\colon  X/\!/\Gamma \to X/\Gamma.
\]   
This map is referred to as the projection of the extended quotient onto the ordinary quotient.
The inclusion 
\[
X\hookrightarrow \widetilde{X}, \quad x \mapsto (e,x)
\]
where $e$ is the identity element of $\Gamma$, is $\Gamma$-equivariant and so passes to quotient spaces to give an inclusion of affine 
varieties $X/\Gamma\hookrightarrow X/\!/\Gamma$. This is referred to as the inclusion 
of the ordinary quotient in the extended quotient. 

Similarly there is a projection map from the quotient space $X\q\Gamma$ to the set of conjugacy classes $c(\Gamma)$ in $\Gamma$. Selecting a set $C$ of representatives for the conjugacy classes, each element of $X\q\Gamma$  can be represented by a pair $(\gamma, x)$ with $\gamma\in C$. The point $x$ lies in the fixed set $X^\gamma$ and is determined up to the action of the centraliser $Z(\gamma)$ of $\gamma$, thus the extended quotient may be decomposed as a disjoint union of components:

\[
X\q\Gamma = \coprod_{\gamma\in C} X^\gamma/ Z(\gamma).
\]

If $X$ is a topological space on which $\Gamma$ acts by homeomorphisms, then the same procedure will create the topological space
$X \q \Gamma$.

\begin{example}
Consider the action of the Coxeter group $\langle s_1, s_2, s_3\mid s_1^2=s_2^2=s_3^2=(s_1s_2)^3=(s_2s_3)^3=(s_3s_1)^3=e\rangle$ on the plane generated by reflections in the sides of an equilateral triangle. The triangle is a strict fundamental domain so the  quotient of the plane under the action can be identified with it. We regard this triangle as a cell complex in the natural way, and note that the stabiliser of any point in the interior of the triangle is trivial, so these orbit points do not ramify in the extended quotient. Points in the interior of an edge of the triangle have stabiliser isomorphic to the cyclic group of order $2$ and so these edges are doubled. The edges corresponding to the conjugacy class of the identity are attached to the triangle (which is also labelled by the identity). The vertices have stabilisers isomorphic to the dihedral group $D_3$ which has three conjugacy classes corresponding to the decomposition of the group into the identity element, the reflections and the rotations, so each vertex ramifies into three vertices. The vertices labelled by the identity element are attached to the triangle, those labelled by the reflection conjugacy class are attached at the endpoints to the two edges labelled by reflections lying in that conjugacy class. The remaining three points remain unattached so the extended quotient  has five components, one closed triangle, one triangular loop and three additional points. 
\end{example}

\section{On cyclic quotients}

Quotients by cyclic group actions play a key role in this paper. In particular we will be concerned with cyclic quotients of real and complex tori, and of complex affine space. The varieties that occur in the latter case were classified in \cite{F} and we refer the reader there for details.  Here we record a useful technical lemma that allows us to simplify the descriptions of the quotients arising in the former case. It allows us to change coordinates on  a torus to push the action entirely onto one factor, and can be regarded as an elementary generalisation of Bezout's identity. 

\begin{lemma}\label{coordinatechange}
Let $G$ be an abelian group and $\zeta\in G$. Let $T=G^n$ and $p_1, \ldots , p_n\in \Z$ so that $T$ is equipped with an action of the cyclic group $\langle \zeta \rangle$ defined by $\zeta\cdot (g_1,\ldots, g_n)=(\zeta^{p_1}g_1, \ldots, \zeta^{p_n}g_n)$. Then there is an  automorphism of $T$ which intertwines the given action with the action $\zeta\diamond(h_1,\ldots, h_n):= (\zeta^d h_1, h_2, \ldots, h_n)$, where $d=\gcd(p_1, \ldots, p_n)$. The automorphism is algebraic in the sense that each coordinate is given by a monomial in the variables.
\end{lemma}

\begin{proof}
For an element $A\in SL_n(\Z)$  we define an automorphism of $T$ as follows. For $g=(g_1,\ldots, g_n)\in T$ define $Ag\in T$ by the formula 

\[
(Ag)_i=\prod\limits_j g_j^{a_{ij}}.
\]

\noindent It is straightforward to verify that this is an action of $SL_n(Z)$ by automorphisms of $T$. Now by the Euclidean algorithm we may choose an element $A\in  SL_n(\Z)$ such that the first column is the (transpose of the) vector $(p_1/d, \ldots, p_n/d)$.

\[
(A(\zeta\diamond g))_i=\prod\limits_{j=1}^n (\zeta\diamond g)_j^{a_{ij}}=(\zeta^d g_1)^{a_{i1}}\prod\limits_{j=2}^n g_j^{a_{ij}}=\zeta^{d a_{i1}}\prod\limits_{j=1}^n g_j^{a_{ij}}=\zeta^{p_i}(Ag)_i.
\]
 So, $A(\zeta\diamond g)=\zeta\cdot(Ag)$ as required.
\end{proof}

\section{Proof of  Theorem \ref{maintheorem}}\label{maintheoremsection}

As usual we identity the conjugacy classes of $W=\fS_n$ with partitions of $n$. Let $\mu$ be a partition of $n$ and let $m_j$ for $j=1,\dots n$ denote the multiplicity of $j$ in the partition $\mu$. We view $\mu$ as the permutation of $1,\ldots, n$ with cycle type $(1)^{m_1}(2)^{m_2}\dots(n)^{m_n}$. Specifically we take $\mu$ to fix the first $m_1$ elements, transpose the next $m_2$ pairs, etc., thus selecting representatives of the conjugacy classes as required. For each $\mu$ we will compute the fixed set and identify the quotient by the centraliser $Z(\mu)$ of $\mu$. 

Let $\bS$ be the standard maximal torus in $\SL_n(\C)$ consisting of diagonal matrices such that the product of the diagonal entries is $1$. Let $\bS_k$ denote the maximal torus $\bS/C_k$ in $\SL_n(\C)/C_k$. An element of $\bS_k$ is fixed by $\mu$ if selecting a representative $s\in \bS$ there exists $\omega\in C_k$ such that $\mu\cdot s=\omega s$. It follows that for a cycle of $\mu$ of length $j$, the corresponding coordinates of $s$ must have the form $a,\omega a, \omega^2 a,\dots \omega^{j-1}a$. Moreover $\omega^{j}a$ must equal $a$ thus $\omega$ must satisfy the equation $\omega^j=1$. So the order of $\omega$ divides the length of each cycle in $\mu$. Since $\omega$ is an element of $C_k$ it follows that this order, written $|\omega|$, divides $h=\gcd(g(\mu),k)$. 

Elements of $\bS_k$ fixed by $\mu$ are thus represented by elements of $\bS$ of the form
$$s=(a_{1,1},\dots,a_{1,m_1},a_{2,1,}\omega a_{2,1},\dots,a_{2,m_2},\omega a_{2,m_2},\dots).$$
We note that some of the multiplicities $m_j$ will be zero, in which case the corresponding string will be empty.

We denote the $C_k$ orbit of $s$ by $[a_{1,1},\dots,a_{1,m_1},a_{2,1}\dots,a_{2,m_2},\dots]$, omitting terms containing a 
power of $\omega$ as these are determined by the others. Note that the total number of coordinates remaining in this notation is the number of parts $c(\mu)$ of $\mu$.

The condition that $s$ lies in $\bS$ implies that
\[
\Omega\Big(\prod_{i=1}^{m_1}a_{1,i}\Big)\Big( \prod_{i=1}^{m_2}a_{2,i}\Big)^2\Big( \prod_{i=1}^{m_3}a_{3,i}\Big)^3\dots=1,
\]
where $\Omega$ is a power of $\omega$. Specifically $\Omega=\omega^\alpha$ where $\alpha={\sum_{j=1}^n m_j\frac{(j-1)j}2}$. Some of the products will be empty; for each of the non-empty products the exponent is divisible by $g$ allowing us to write the formula as
\begin{equation}
\left(\Big(\prod_{i=1}^{m_1}a_{1,i}\Big)^{1/g}\Big( \prod_{i=1}^{m_2}a_{2,i}\Big)^{2/g}\Big( \prod_{i=1}^{m_3}a_{3,i}\Big)^{3/g}\dots\right)^g=\Omega^{-1}.
\label{eq: omega}
\end{equation}

The permutation $\mu$ is the product of disjoint cycles $\mu=\tau_{1,1}\dots\tau_{1,m_1}\tau_{2,1}\dots\tau_{2,m_2}\dots$ where each $\tau_{j,i}$ is a $j$-cycle.  The centraliser $Z(\mu)$ is the group generated by the cycles $\tau_{j,i}$ along with permutation groups $\fS_{m_1},\fS_{m_2},\dots$ with $\fS_{m_j}$ acting by permuting the $m_j$ parts of size $j$. Hence $\fS_{m_j}$ simply permutes the coordinates $a_{j,1},\dots,a_{j,m_j}$.

The action of $\tau_{j,i}$ on $[a_{1,1},\dots,a_{1,m_1},a_{2,1}\dots,a_{2,m_2},\dots]$ has the effect of multiplying the coordinate $a_{j,i}$ by $\omega$ while leaving all other coordinates fixed.

For each $j,l$ let $\sigma_{j,l}$ be the $l$-th degree symmetric polynomial in the variables $a_{j,1}^{|\omega|}, \ldots, a_{j,m_j}^{|\omega|}$. By construction the actions of $\tau_{j,i}$ and $\fS_{m_j}$ on the coordinates leave $\sigma_{j,l}$ fixed for all $j,l$.

Conversely suppose that $(a_{j,i})$ and $(a'_{j,i})$ are coordinates yielding the same values of $\sigma_{j,l}$ for all $j,l$. It follows that the powers $a_{j,i}^{|\omega|}$ and $(a'_{j,i})^{|\omega|}$ agree up to a permutation (for each $j$) of the $i$ indices. Hence multiplying each $a_{j,i}$ by some power of $\omega$ and permuting we obtain $a'_{j,i}$. So the coordinates are identified by the action of the centraliser $Z(\mu)$ precisely when they yield the same values of $\sigma_{j,l}$.

The action of the generator $\zeta=e^{2\pi i/k}$ of $C_k$ on the coordinates has the effect of multiplying $\sigma_{j,l}$ by $\zeta^{l|\omega|}$.  Note that $\zeta^{k/|\omega|}$ acts trivially on all the coordinates, hence it is really an action of $C_{k/|\omega|}$.

To summarise we have now identified the quotient of the $\mu$-fixed set by the centraliser $Z(\mu)$ with a subspace of $\A^{c(\mu)}/C_{k/|\omega|}$.  We denote a point of the image by $[\sigma_{j,l}]$.

In order to identify the subspace we reformulate Equation \ref{eq: omega} in terms of the variables $\sigma_{j,l}$. We have
\begin{equation}
\left(\prod_{j} \sigma_{j,m_j}^{j/g}\right)^{g/|\omega|}=\Omega^{-1}.
\label{eq: sigma}
\end{equation}
The action of the generator $\zeta$ has the effect of multiplying $\prod_{j} \sigma_{j,m_j}^{j/g}$ by $\zeta^q$ where $q=\sum\limits_j |\omega|jm_j/g=|\omega|n/g$. Now $\zeta^q$ generates a group of order
$$\frac k{\gcd(k,q)}=\frac{kg}{\gcd(kg,n|\omega|)}=\frac{g}{\gcd(g,k'|\omega|)}$$
where $k'=n/k$.

By Equation \ref{eq: sigma} the product $\prod_{j} \sigma_{j,m_j}^{j/g}$ can take $g/|\omega|$ values, however factoring out by the (free) action of $\zeta^q$ gives 
$$\frac g{|\omega|} \cdot \frac{\gcd(g,k'|\omega|)}g=\gcd(g/|\omega|,k')$$
components. Let $\xi$ be one of the $g/|\omega|$ roots of $\Omega^{-1}$ and consider the variety defined by the equation
\begin{equation}\label{subvariety}
\prod_{j} \sigma_{j,m_j}^{j/g}=\xi.
\end{equation}
In these $b(\mu)$ variables this equation defines a complex torus of codimension $1$. However there are a further $c(\mu)-b(\mu)$ variables $\sigma_{j,l}$ with $l<m_j$ which are unconstrained. So we obtain a variety of the form  $(\C^\times)^{b(\mu)-1}\times \A^{c(\mu)-b(\mu)}$, and 
the component of the extended quotient corresponding to this variety is given by factoring out the action of the subgroup $C_{\gcd(k,q)}=\langle \zeta^{k/\gcd(k,q)}\rangle$. We recall that this generator acts on each coordinate $\sigma_{j,l}$ of the the variety as multiplication by $\zeta^{l |\omega| k/\gcd(k,q)}$. 

First we will consider the action on the constrained coordinates $\sigma_{j,m_j}$. Here $\zeta^{k/\gcd(k,q)}$ acts by multiplication by $\zeta^{m_j |\omega| k/\gcd(k,q)}$, and we note that this is the restriction of the natural multiplication action on the larger variety $(\C^\times)^b$. Following the proof of  Lemma  \ref{coordinatechange}  we select an element $A\in SL_b(\Z)$ with first column equal to the transpose of the vector $\left (\frac{m_{j_1}}{m(\mu)}, \ldots, \frac{m_{j_b}}{m(\mu)}\right)$, where the indices $j_i$ denote the distinct parts of $\mu$. Abusing notation we will use $j$ to denote an index in the set $J=\{j_1, \ldots, j_b\}$.

We now transform the coordinates $\sigma_{j,m_{j}}$ by the inverse matrix $A^{-1}$ to obtain new coordinates which we denote by $\rho_i$. Now (again abusing notation) the coordinates $\sigma_{j,m_j}$ are recovered by the formula $\displaystyle \sigma_{j,m_j}= \prod_{i=1}^b \rho_i^{a_{ji}}$. 

The subvariety defined by equation \ref{subvariety} is transformed by $A^{-1}$ to the variety

\begin{equation}\label{newvariety}
\prod_{i=1}^b\prod_{j\in J} \left(\rho_i^{a_{ji}}\right)^{j/g}=\prod_{i=1}^b  \rho_i^{\sum_{j\in J}a_{ji}j/g}=\xi.
\end{equation}
which we denote by $V_{\mu,\omega,\xi}$. Recall that there are still $c(\mu)-b(\mu)$ unconstrained coordinates $\sigma_{j,l}, j<m_j$ in this variety for which we have not changed variables.

Now let $d=\frac{m(\mu)|\omega|k}{\gcd(k,q)}$. According to the Lemma, in these coordinates the action of the generator multiplies the first coordinate $\rho_1$ by $\zeta^{d}$ and leaves $\rho_2, \ldots, \rho_b$ invariant. Let $\lambda$ denote the order of $\zeta^{d}$. i.e., $\lambda=\frac{k}{\gcd(d,k)}$. We will show that the quotient $\kappa:=\frac{\sum_{j\in J}a_{j1}j/g}{\lambda}$ is an integer which will allow us to rewrite Equation \ref{newvariety} in the form

\begin{equation}\label{newervariety}
(\rho_1^\lambda)^\kappa \cdot \prod_{i=2}^b  \rho_i^{\sum_{j\in J}a_{ji}j/g}=\xi.
\end{equation}
where $\rho_1^\lambda$ is left invariant by the action.

First we compute
\[
\lambda= \frac{k}{\gcd\left(\frac{m(\mu)|\omega| k}{\gcd(k,q)},k\right)}=\frac{\gcd(k,q)}{\gcd\left(m(\mu)|\omega|,\gcd(k,q)\right)}=\frac{\gcd(k,q)}{\gcd\left(m(\mu)|\omega|,k,q\right)}=\frac{\gcd(k,q)}{\gcd\left(m(\mu)|\omega|,k\right)},
\]
since $m(\mu)|\omega|$ divides $q=\frac{n}{g}|\omega|$.

On the other hand, the exponent of $\rho_1$ is $\sum\limits_j a_{j1}j/g=\sum\limits_j \frac{m_{j}}{m(\mu)}\frac{j}{g}=\frac{n}{m(\mu)g}$. Thus 
\[
\kappa=\frac{n \gcd(m(\mu)|\omega| ,k)}{m(\mu)g \gcd(k,q)}=\frac{\gcd(n m(\mu)|\omega|, nk)}{\gcd(m(\mu)gk, m(\mu)gq)}=\frac{\gcd(nk, n m(\mu)|\omega|)}{\gcd(m(\mu)gk, nm(\mu)|\omega|)},
\]
which is an integer because $m(\mu)g$ divides $n$.

\bigskip
To recap, we have changed coordinates on the constrained variables replacing $\sigma_{j,m_j}$ by $\rho_i$, and pushing the action there entirely onto the first coordinate $\rho_1$, $\rho_1\mapsto \zeta^d\rho_1$.

We will now  consider the unconstrained variables $\sigma_{j,l}$ where $l < m_j$, where the action is (still) given  by $\sigma_{j,l}\mapsto \zeta^{l |\omega| k/\gcd(k,q)}\sigma_{j,l}$. Let 
\[
\eta=  \zeta^{\lambda|\omega| k/\gcd(k,q)}= \zeta^{k|\omega|/\gcd(m(\mu)|\omega|,k)}=\zeta^{k/\gcd(m(\mu),k/|\omega|)}
\] which generates the cyclic group $C_{\gcd(m(\mu),k/|\omega|)}$. This acts on the affine space of dimension $c(\mu)-b(\mu)$ spanned by the unconstrained variables, with the generator $\eta$ acting by $\sigma_{j,l}\mapsto \eta^l\sigma_{j,l}$. The cyclic singularity $\A_{\mu,\omega}$ is defined to be the quotient variety.  It is easy to see that for each $l$ there are $p_l(\mu)$ coordinates on which the action multiplies by $\eta^l$.

We will next define a map $\phi$ from the variety $V_{\mu, \omega, \xi}\subset (\C^\times)^b\times \A^{c-b}$ to $(\C^\times)^b\times \A_{\mu, \omega}$, with image a codimension 1 subvariety.

Recall that $\lambda = \frac{k}{\gcd(d,k)}$ so by Bezout's Lemma there exists an integer $\delta$ such that 

\begin{equation}\label{congruencemodlambda}
\frac{d}{\gcd(d,k)}\cdot \delta \equiv 1 \mod \lambda.
\end{equation}
Let $\rhohat$ denote a $\gcd(d,k)$-th root of $\rho_1$. While this is not uniquely defined, we will see that the construction that follows is independent of the choice. Set $\alpha_l=\dfrac{l |\omega| k\delta(\lambda-1)}{\gcd(k,q)}$, and define the map $\phi$ by 

\[
\Big((\rho_1, \rho_2,\ldots,\rho_b),(\sigma_{j,l})_{j\in J, l<m_j}\Big)\mapsto \Big((\rho_1^\lambda, \rho_2,\ldots,\rho_b),(\rhohat^{\alpha_l}\sigma_{j,l})_{j\in J, l<m_j}\Big).
\]

To see that this is well defined we need to verify that it is independent of the choice of $\rhohat$, but $\rhohat$ is defined up to multiplication by a $\gcd(d,k)$-th root of unity, all of which are powers of the primitive root $\zeta^\lambda$, so it suffices to show that replacing $\rhohat$ by $\zeta^\lambda\rhohat$ does not change the image. 

We have 

\[
\zeta^{\lambda\alpha_l}= \left(\zeta^{\frac{\lambda l |\omega| k}{\gcd(k,q)}}\right)^{\delta(\lambda-1)}=\left(\eta^l\right)^{\delta(\lambda-1)}
\]
Hence as points in the quotient space $\A_{\mu,\omega}$,
\[
\left((\zeta^\lambda\rhohat)^{\alpha_l}\sigma_{j,l}\right)_{j\in J, l<m_j}=\left(\left(\eta^l\right)^{\delta(\lambda-1)}\rhohat^{\alpha_l}\sigma_{j,l}\right)_{j\in J, l<m_j}=\Big(\rhohat^{\alpha_l}\sigma_{j,l}\Big)_{j\in J, l<m_j}.
\]

Now we will show that $\phi$ is constant on orbits of the action of  $C_{\gcd(k,q)}=\langle \zeta^{k/\gcd(k,q)}\rangle$ of $V_{\mu,\omega,\xi}$. We have
\[
\phi\Big((\zeta^d\rho_1, \rho_2,\ldots,\rho_b),(\zeta^{l |\omega| k/\gcd(k,q)}\sigma_{j,l})\Big) = \Big((\zeta^{d\lambda}\rho_1^\lambda, \rho_2,\ldots,\rho_b),((\zeta^{\frac{d}{\gcd(d,k)}}\rhohat)^{\alpha_l}\zeta^{l |\omega| k/\gcd(k,q)}\sigma_{j,l})\Big)
\]
Now $d\lambda=\frac{dk}{\gcd(d,k)}$ which is divisible by $k$ so $\zeta^{d\lambda}=1$. Also we compute

\[
\frac{d}{\gcd(d,k)}\alpha_l+\frac{l |\omega| k}{\gcd(k,q)}=\left(\frac{d\delta(\lambda-1)}{\gcd(d,k)}+1\right)\frac{l |\omega| k}{\gcd(k,q)}.
\]

Using Equation \ref{congruencemodlambda} we now observe that $\displaystyle \frac{d\delta(\lambda-1)}{\gcd(d,k)}+1\equiv 0, \mod \lambda$ and set $r=\frac{1}{\lambda}\left(\frac{d\delta(\lambda-1)}{\gcd(d,k)}+1\right)$. It then follows that $\zeta^{\frac{d\alpha_l}{\gcd(d,k)}}\zeta^{l |\omega| k/\gcd(k,q)}=\eta^{rl}$, yielding
\begin{align*}
\phi\Big((\zeta^d\rho_1, \rho_2,\ldots,\rho_b),(\zeta^{l |\omega| k/\gcd(k,q)}\sigma_{j,l})\Big) &= \Big((\rho_1^\lambda, \rho_2,\ldots,\rho_b),(\eta^{rl}\rhohat^{\alpha_l}\sigma_{j,l})\Big)\\
&=\Big((\rho_1^\lambda, \rho_2,\ldots,\rho_b),(\rhohat^{\alpha_l}\sigma_{j,l})\Big)\\
&=\phi\Big((\rho_1, \rho_2,\ldots,\rho_b),(\sigma_{j,l})\Big)
\end{align*}
in $(\C^\times)^b\times \A_{\mu,\omega}$, as required.

We conclude that $\phi$ induces a map from the quotient of $V_{\mu,\omega,\xi}$ by the action of $C_{\gcd(k,q)}$ to the variety $(\C^\times)^b\times \A_{\mu,\omega}$. We will now show that this induced map is injective.

Suppose then that $
\Big((\rho_1, \rho_2,\ldots,\rho_b),(\sigma_{j,l})\Big)$ and $
\Big((\rho'_1, \rho'_2,\ldots,\rho'_b),(\sigma'_{j,l})\Big)$ have the same image under $\phi$. Then $\rho_1$ and $\rho'_1$ differ by a a power of $\zeta^d$, noting that this is a primitive $\lambda$-th root of unity. It follows that up to the action of $C_{\gcd(k,q)}$ the first coordinates agree. Moreover we have noted that applying the action does not change the image under $\phi$, so without loss of generality we may assume that $\rho_i'=\rho_i$ for all $i$. We will  show that the $\sigma_{j,l}$ coordinates must now agree up to the action of the subgroup of $C_{\gcd(k,q)}$ which stabilises the $\rho$ coordinates.

We are given that the coordinates $\rhohat^{\alpha_l}\sigma_{j,l}$ agree with $\rhohatprime^{\alpha_l}\sigma'_{j,l}$ up to the action of $\langle \eta\rangle$, i.e., for some $r$
\[
\rhohatprime^{\alpha_l}\sigma'_{j,l}=\eta^{rl}\rhohat^{\alpha_l}\sigma_{j,l}
\]

But since $\rho_1=\rho_1'$, we may choose the roots $\rhohat, \rhohatprime$ to be equal, giving 
\[
\sigma'_{j,l}=\eta^{rl}\sigma_{j,l}=\zeta^{rl\lambda|\omega| k/\gcd(k,q)}\sigma_{j,l}=(\zeta^{l|\omega| k/\gcd(k,q)})^{r\lambda}\sigma_{j,l}.
\]
But then, as $\zeta^{dr\lambda}=1$, we have
\[
(\zeta^{\frac{k}{\gcd(k,q)}})^{r\lambda}\cdot \Big((\rho_1, \rho_2,\ldots,\rho_b),(\sigma_{j,l})\Big)=
\Big((\rho'_1, \rho'_2,\ldots,\rho'_b),(\sigma'_{j,l})\Big)
\]
so the map induced by $\phi$ is injective as required.

Finally we consider the image of the induced map. Given an element of $(\C^\times)^b\times \A_{\mu,\omega}$ represented by 

\[
\Big((\psi_1, \psi_2,\ldots,\psi_b),(\tau_{j,l})\Big)
\]
choose a $\lambda$-th root $\rho_1$ of $\psi_1$ and a $\gcd(d,k)$-th root $\rhohat$ of $\rho_1$. We set $\rho_i=\psi_i$ for $i\geq 2$ and $\sigma_{j,l}= \rhohat^{-\alpha_l}\tau_{j,l}$. Then $\Big((\rho_1, \rho_2,\ldots,\rho_b),(\sigma_{j,l})\Big)$ is in the variety $V_{\mu,\omega,\xi}$ if and only if Equation \ref{newervariety} is satisfied by these coordinates.  When this happens the image is precisely $\Big((\psi_1, \psi_2,\ldots,\psi_b),(\tau_{j,l})\Big)$. Equation \ref{newervariety} translates in these coordinates into the equation
\begin{equation}\label{newestvariety}
\psi_1^\kappa \cdot \prod_{i=2}^b  \psi_i^{\sum_{j\in J}a_{ji}j/g}=\xi.
\end{equation}
Since the variables $\tau_{j,l}$ are unconstrained the image is the product of a codimension-1 subvariety $T$ of $(\C^\times)^b$ with the cyclic singularity $\A_{\mu,\omega}$. 

Finally we will show that the exponents appearing in Equation \ref{newestvariety} are coprime, from which it follows that the subvariety $T$ is a single torus. Recall that $g$ is the greatest common divisor of the parts $j_i$ of the partition. We form a matrix $B$ in $SL_b(\Z)$ with first column row given by $b_{1i}=j_i/g$, which is possible since the numbers $j_i/g$ are coprime. Then the first row of the product $BA$ has entries $(BA)_{1i}=\sum\limits_j a_{ji}j/g$. Since the matrix $BA$ has determinant $1$ these entries are coprime. Dividing the first entry by $\lambda$ to get the exponent $\kappa$ of $\psi_1$ does not change the greatest common divisor. This completes the proof.

\section{Proof of Corollary \ref{corollary}}\label{proof of corollary}

The cyclic singularities are contractible so to prove Corollary \ref{corollary} it suffices to show that 
\[
\left|\coprod\limits_{\omega\in C_{\gcd(g(\mu), k)}}\   X_{\mu,\omega}\right|=\displaystyle \frac{g(\mu)}a\sum_{s=0}^{a-1} \gcd(a,s)
\]
 where $\displaystyle a=\gcd\left(g(\mu),\frac n {g(\mu)}, k,\frac n k\right)$.

There are $h=\gcd(g,k)$ possible values for $\omega$. Writing $\omega$ as the $p$th power of the generator of $C_h$, the order of $\omega$ is $h/\gcd(h,p)$ so, setting $k'=n/k$,  the cardinality is
\begin{align*}
\sum_{\omega\in C_h} \gcd(g/|\omega|, k')&=\sum_{p=0}^{h-1} \gcd\left(\frac{g\gcd(h,p)}h,k'\right)\\
&=\sum_{p=0}^{h-1} \gcd(\gcd(g,gp/h),k')=\sum_{p=0}^{h-1} \gcd(g,gp/h,k').
\end{align*}
Letting $h'=\gcd(g,k')$ and $a=hh'/g$ we observe that
\[
a=\frac{\gcd(g,k)\gcd(g,k')}g=\frac{\gcd(g^2,gk,gk',kk')}g=\gcd(g,k,k',n/g).
\]
The cardinality is now given by
\begin{align*}
\sum_{p=0}^{h-1} \gcd(h',gp/h)&=\sum_{p=0}^{h-1} \gcd(h',ph'/a)=\frac{h'}a\sum_{p=0}^{h-1} \gcd(a,p)\\
&=\frac{h'h}{a^2}\sum_{s=0}^{a-1} \gcd(a,s)=\frac{g}a\sum_{s=0}^{a-1} \gcd(a,s).
\end{align*}
The third equality follows from the fact that $\gcd(a,p)$ repeats with period $a$, since $\gcd(a,p+a)=\gcd(a,p)$.

We remark that $a$ is symmetric in the interchange of $k$ with $k'=n/k$ giving the same number of components in each stratum for a group and its Langlands dual.

\section{Proof of Theorem \ref{maintheorem2}}\label{proof of main2}

Let $T$ denote the standard maximal torus in $SU_n(\C)$ consisting of diagonal matrices with entries of modulus $1$ and with determinant $1$. Let $T_k$ denote the maximal torus $T/C_k$ in $SU_n(\C)/C_k$.  The Weyl group is again the permutation group $W=\fS_n$ and conjugacy classes of elements correspond to partitions of $n$.

The fixed set in the maximal torus $T_k=T/C_k$ corresponding to a partition $\mu$ is a subspace of the fixed set for the action on $\bS_k=\bS/C_k$ considered in Section \ref{maintheoremsection}. Specifically the fixed set in $T_k$ is the subspace of the fixed set in $\bS_k$ where each coordinate is required to have modulus $1$.

Recall that the points of the fixed set in $\bS_k$ are parametrised by elements $\omega\in C_k$, with order dividing $h=\gcd(g,k)$, along with tuples of variables $(a_{j,i})$ where $j$ ranges over the sizes of parts of the partition and $i$ ranges from $1$ to $m_j$, the multiplicity of $j$ in the partition. These must satisfy the equation
\[
\Big(\prod_{i=1}^{m_1}a_{1,i}\Big)^{1}\Big( \prod_{i=1}^{m_2}a_{2,i}\Big)^{2}\Big( \prod_{i=1}^{m_3}a_{3,i}\Big)^{3}\dots=\Omega^{-1}
\]
where $\Omega=\omega^{\sum_{j=1}^n m_j\frac{(j-1)j}2}$, see Equation \ref{eq: omega}. We now require one additional constraint per variable, namely $|a_{j,i}|=1$. The coordinates are not uniquely determined as we must factor out the action of $C_k$: changing each $a_{j,i}$ by a factor of $\zeta=e^{2\pi \sqrt{-1}/k}$ gives the same point of the quotient.

It is convenient to introduce polar coordinates $\theta_{j,i}$ such that $a_{j,i}=e^{2\pi\sqrt{-1}\, \theta_{j,i}}$. Since $i$ is used as an index, we will not use it to denote the square root of $-1$.  Taking $\Theta$ such that $\Omega=e^{2\pi\sqrt{-1}\, \Theta}$ the above equation yields
\begin{equation}
\sum_{j}\sum_{i=1}^{m_j} j\, \theta_{j,i} = -\Theta\mod \Z
\label{eq: Theta}
\end{equation}
where $j$ ranges over the parts of $\mu$. We denote the space of points $(\theta_{j,i})$ satisfying this equation by $V_{\mu,\omega}$.

For each $\mu$, we must identify the quotient of this fixed set by the centraliser $Z(\mu)$ of this permutation.  Recall that the centraliser is generated by cycles $\tau_{j,i}$ which act by multiplying $a_{j,i}$ by $\omega$, and by the symmetric groups $\fS_{m_j}$ which act by permuting the coordinates $(a_{j,1},\dots,a_{j,m_j})$. 

Lifting from the variables $a_{j,i}$ to $\theta_{j,i}$ has the effect of `lifting' the centraliser as follows. The variables $\theta_{j,i}$ are determined by $a_{j,i}$, only modulo $\Z$. The group $C_k$ acts by multiplying each $a_{j,i}$ by $\zeta=e^{2\pi\sqrt{-1}\, \frac 1k}$ and hence the action lifts to an action of the infinite cyclic group $\Z$, where the generator, which we denote $\zetilde$ acts as a shift by $\frac1k$ on each $\theta_{j,i}$. 
It is easy to see that the cyclic group $\langle \omega \rangle$ is generated by $e^{2\pi \sqrt{-1}\, \frac 1{|\omega|}}$, thus the variables $\theta_{j,i}$ can additionally be shifted by any integer multiple of $\frac 1{|\omega|}$. Permuting the variables $a_{j,i}$ simply corresponds to permuting $\theta_{j,i}$. Hence the action of the centraliser corresponds to the action on the $\theta_{j,i}$ coordinates generated by the lattice with coordinates in $\frac 1{|\omega|}\Z$, the shift $\zetilde$, and by these permutations.  We will denote this group by $\widetilde{Z}(\mu,\omega)$.

We now introduce variables $\sigma_j=\sum_{i=1}^{m_j} \theta_{j,i}|\omega|$. These correspond to the maximal degree symmetric polynomials in the complex case.   We rewrite Equation \ref {eq: Theta} in the form
\begin{equation*}
\sum_{j}\frac{j}{|\omega|} \sigma_{j} = -\Theta\mod \Z
\end{equation*}
where the sum is taken over the distinct parts $j$ of $\mu$. The action of $\zetilde$ has the effect of shifting $\sigma_j$ by $\frac {m_j|\omega|}k$ and hence changes $\sum_j \frac j{|\omega|}\sigma_j$ by
$$\sum_j  \frac j{|\omega|}\,\frac {m_j|\omega|}k=\frac{n}{k}.$$
Note that shifting the value of a $\theta_{j,i}$ by $\frac 1{|\omega|}$ has the effect of shifting the value $\sigma_j$ by $1$ and hence changing $\sum_j \frac j{|\omega|}\sigma_j$ by $\frac{j}{|\omega|}$.

It follows that the sum $\sum_j \frac j{|\omega|}\sigma_j$ can be shifted by any multiple of $\gcd( \frac nk, \frac g{|\omega|})$ where $g$ is the greatest common divisor of the parts $j$.  Thus we can assume without loss of generality that $\sum_{j}\frac{j}{|\omega|} \sigma_{j}$ lies in the finite set
$$X_{\mu,\omega}=\Big\{-\Theta,-\Theta+1,\dots,-\Theta+\gcd\Big( \frac nk, \frac{g}{|\omega|}\Big)-1\Big\}.$$
Indeed as every element of $\widetilde{Z}(\mu,\omega)$ shifts the sum by a multiple of $\gcd( \frac nk, \frac{g}{|\omega|})$, these different values give $\gcd( \frac nk, \frac{g}{|\omega|})$ distinct components of the quotient.

Now fixing the value of $\sigma_j$ we consider tuples $\theta_{j,1},\dots, \theta_{j,m_j}$ such that the sum $\sum_{i=1}^{m_j}\theta_{j,i}|\omega|$ yields this value. We use Morton's description of symmetric products of circles, see \cite{M}. We will consider the subgroup $\widetilde{Z}_0(\mu,\omega)$ of $\widetilde{Z}(\mu,\omega)$ generated by the sublattice
\[
\{(\nu_{j,i}) : \nu_{j,i}\in \frac{1}{|\omega|}\Z \text{ for all $j,i$ and } \nu_{j,1}+\dots+\nu_{j,m_j}=0\text{ for all $j$}\}
\]
along with the permutation groups $\fS_{m_j}$. Note that this subgroup preserves each of the variables $\sigma_j$. By adding integer multiples of $\frac 1{|\omega|}$ (totalling zero) to the variables $\theta_{j,i}$ we can assume that the minimum and maximum values differ by at most $\frac 1{|\omega|}$. Additionally, the action of the symmetric group allows us to arrange the variables $\theta_{j,i}$ in ascending order, thus we may assume that
\begin{equation}
\theta_{j,1}\leq \theta_{j,2}\leq\dots \leq \theta_{j,m_j}\leq \theta_{j,1}+\frac1{|\omega|}.
\label{eq: Morton}
\end{equation}
Moreover this condition yields a unique representative for each orbit of the group $\widetilde{Z}_0(\mu,\omega)$.

For $x\in X_{\mu,\omega}$ we denote by $V_{\mu,\omega,x}$ the subspace of $V_{\mu,\omega}$ defined by the equation
\begin{equation}
\sum_{j}\frac{j}{|\omega|} \sigma_{j}=x
\label{eq: V_mu,omega,x}
\end{equation}
along with Morton's inequalities \ref{eq: Morton}. For each $j$, we observe that for a fixed value of $\sigma_j$, the set of points $(\theta_{j,1},\dots,\theta_{j,m_j})$ satisfying the inequalities form a simplex of dimension $m_j-1$. The variables $\sigma_j$ can take any real values, subject to the constraint \ref{eq: V_mu,omega,x}, hence these lie in a codimension $1$ affine subspace of $\R^b$. It follows that the space $V_{\mu,\omega,x}$ is a product of $\R^{b-1}$ with a polysimplex whose component simplices have dimensions $m_j-1$. We will think of $V_{\mu,\omega,x}$ as a bundle of polysimplices over the space $\R^{b-1}$.

We now consider elements of $\widetilde{Z}(\mu,\omega)$ which preserve Morton's inequalities, \ref{eq: Morton}. Let $W_j$ denote the element defined by
\[
(\theta_{j,1},\dots,\theta_{j,m_j})\mapsto (\theta_{j,2},\dots \theta_{j,m_j}, \theta_{j,1}+\frac 1{|\omega|}).
\]
This shifts the sum $\sum_{i=1}^{m_j}\theta_{j,i}$ by $\frac 1{|\omega|}$ and hence shifts $\sigma_j$ by $1$, while leaving all other $\sigma_{j'}$ invariant, and preserves Morton's inequalities. Note that $\zetilde$ also preserves these inequalities and that the elements $W_j$ along with $\zetilde$ and $\widetilde{Z}_0(\mu,\omega)$ generate the whole of $\widetilde{Z}(\mu,\omega)$.

Each $W_j$ changes $\sum_{j}\frac{j}{|\omega|} \sigma_{j}$ by $\frac{j}{|\omega|}$ and hence takes this value out of the set $X_{\mu,\omega}$.
Similarly the shift $\zetilde$ changes the sum by $\frac{n}{k}$ so we must consider compositions of the form $\zetilde^\gamma\prod_j W_j^{\beta_j}$
where
$$\gamma \frac{n}{k}+\sum_j \frac{j}{|\omega|}\beta_j=0.$$
Let $L_{\mu,\omega}$ denote the group of elements of this form.
The quotient by $\widetilde{Z}(\mu,\omega)$ of the space $V_{\mu,\omega}$ is thus the disjoint union
\[
\coprod_{x\in X_{\mu,\omega}} V_{\mu,\omega,x}/L_{\mu,\omega}
\]
We remark that the elements $\zetilde, \{W_j\}$ are commuting elements of infinite order, so $L_{\mu,\omega}$ is an abelian group of rank $b-1$,
indeed $L_{\mu,\omega}$ can be identified as a quotient of the lattice
\[
\{(\gamma,(\beta_j))\in \Z\times\Z^b\,:\, \gamma \frac{n}{k}+\sum_{j}\frac{j}{|\omega|} \beta_j=0\}
\]
of rank $b$, where the kernel is the cyclic subgroup generated by the element
\[
\gamma=\frac k{|\omega|},\quad \beta_j=-m_j.
\]
It is easy to see that this is in the kernel and moreover for the linear part of the affine map $W_j^{\beta_j}$ to be the identity $m_j$ must divide $\beta_j$.

We will now compute the stabiliser of the fibres of $V_{\mu,\omega,x}$, i.e.\ we will determine which elements $\zetilde^\gamma\prod_j W_j^{\beta_j}$ preserve all the $\sigma_j$ variables.  This will hold when
$$\frac{\gamma m_j|\omega|}k +\beta_j=0$$
for each $j$. The elements $(\gamma,(\beta_j))$ satisfying this equation form an infinite cyclic group whose generator is given by
\[
\gamma=\frac{k}{\gcd(m|\omega|,k)},\quad \beta_j=-\frac{m_j}{\gcd(m,k/|\omega|)}.
\]
The action of $\zetilde^\gamma\prod_j W_j^{\beta_j}$ on each simplex factor of the fibre is then the $\beta_j$ power of the $m_j$-cycle on the vertices. This decomposes as a product of $\frac{m_j}{\gcd(m,k/|\omega|)}$ disjoint $\gcd(m,k/|\omega|)$-cycles on each simplex. The effective action is thus by a group of order $\gcd(m,k/|\omega|)$ and each simplex can be regarded as a join of $m_j/\gcd(m,\frac{k}{|\omega|})$ simplices of dimension $\gcd(m,\frac{k}{|\omega|})-1$ on each of which this group acts by cyclically permuting the vertices.

The action on each simplex of the join is orientable or not depending on whether or not the dimension $\gcd(m,\frac{k}{|\omega|})-1$ is even. It follows that the action on the join preserves orientation precisely when $(\gcd(m,\frac{k}{|\omega|})-1)m_j/\gcd(m,\frac{k}{|\omega|})$ is even and that the action on the polysimplex preserves the orientation when
\[
\sum_j \frac{(\gcd(m,\frac{k}{|\omega|})-1)m_j}{\gcd(m,\frac{k}{|\omega|})}=\sum_j m_j\Big(1-\frac{1}{\gcd(m,\frac{k}{|\omega|})}\Big)=c-\frac{c}{\gcd(m,\frac{k}{|\omega|})}
\]
is even. Hence this action preserves the orientation on the fibres if and only if $c$ is odd or the $2$-adic norms satisfy $|c|_2<|\gcd(m,\frac{k}{|\omega|})|_2$.

To complete the proof we observe that $L_{\mu,\omega}$ is the product of a free abelian group $\Gamma_{\mu,\omega}$ of rank $b-1$ with the cyclic group $C_{\gcd(m,\frac{k}{|\omega|})}$ and define $E_{\mu,\omega}=V_{\mu,\omega,x}/\Gamma_{\mu,\omega}$. Note that up to a homeomorphism induced by translation, this is independent of $x\in X_{\mu,\omega}$.  The quotient $E_{\mu,\omega}$ is a bundle of polysimplices over a torus of dimension $b-1$

Thus we have
\[
T\q W \cong \coprod_\mu\coprod_{\omega\in C_{\gcd(g(\mu), k)}}\coprod_{x\in X_{\mu,\omega}} V_{\mu,\omega,x}/L_{\mu,\omega} \cong \coprod_\mu\coprod_{\omega\in C_{\gcd(g(\mu), k)}}\coprod_{x\in X_{\mu,\omega}}  E_{\mu,\omega}/C_{\gcd(m(\mu),\frac{k}{|\omega|})}
\]
where $C_{\gcd(m(\mu),\frac{k}{|\omega|})}$ acts on the fibres of the bundle as described above.

It follows that each component is a bundle over a compact torus of dimension $b-1$ with fibre a cyclic quotient of a polysimplex.

\bigskip

It is instructive to consider the special case when $k=1$, for which $\omega$ can only take the value $1$.  In this case the components are the bundles $E_{\mu,1}$. These bundles are obtained by gluing polysimplex fibres using the compositions $\prod_j W_j^{\beta_j}$ where $\sum_{j} j \beta_j=0$. In this context we have the following simple lemma in the spirit of Morton \cite{M}. Indeed in the case where we have a partition of the form $1+j+\dots+j$ for some $j>1$, the space $E_{\mu,1}$ is simply a symmetric product of $m_j$ circles. Morton showed that this is a simplex bundle over a circle, which is orientable or not according to whether $m_j$ is odd or even.

\begin{lemma}
In the case $k=1$, given a partition $\mu$ of $n$ with distinct parts $j_1,\dots,j_b$ and $g=\gcd(j_1,\dots,j_b)$, the bundle $E_{\mu,1}$ is non-orientable if and only if the vectors $(j_1/g,\dots,j_b/g)$ and $(m_{j_1}-1,\dots,m_{j_b}-1)$ are linearly independent as elements of $(\Z/2)^b$.
\end{lemma}
\begin{proof}
The composition $\prod_j W_j^{\beta_j}$ preserves the orientation on the polysimplex if and only if $\sum_{j} (m_j-1)\beta_j$ is even. If the vectors are linearly dependent then (noting that $(j_1/g,\dots,j_b/g)$ is non-zero modulo $2$ since its entries have greatest common divisor $1$) we have either $(m_{j_1}-1,\dots,m_{j_b}-1)=(0,\dots,0)$ modulo $2$, whence each $W_j$ preserves the fibre orientation, or $(m_{j_1}-1,\dots,m_{j_b}-1)=(j_1/g,\dots,j_b/g)$ modulo $2$. In the latter case any element $(\beta_{j_1},\dots\beta_{j_b})$ in the lattice satisfies $\sum_{j}\frac{j}{g} \beta_j=0$ hence
\[
0=\sum_{j}\frac{j}{g} \beta_j\cong\sum_{j} (m_j-1)\beta_j\mod 2
\]
and again all elements of the lattice preserve orientation on the fibres.

Conversely suppose $(j_1/g,\dots,j_b/g)$ and $(m_{j_1}-1,\dots,m_{j_b}-1)$ are linearly independent modulo $2$. Then in particular there exists a pair $j,j'$ such that $(j/g,j'/g)$ and $(m_{j}-1,m_{j'}-1)$ are linearly independent modulo $2$. Set $\beta_j=j'/g, \beta_{j'}=-j/g$ and $\beta_{j''}=0$ for all other $j''$. This defines a point of the lattice and the composition $W_j^{\beta_j}W_{j'}^{\beta_{j'}}$ will reverse orientation on the polysimplex fibres since the independence of $(j/g,j'/g)$ and $(m_{j}-1,m_{j'}-1)$ modulo $2$ ensures that $(m_{j}-1)\frac {j'}g-(m_{j'}-1)\frac {j}g$ is odd.
\end{proof}

\begin{example} We consider in detail the real counterpart $SU_6(\C)$ of the example $SL_6(\C)$ considered in section \ref{examples}.
\[
\begin{array}{c|c|c|c|c|c}
 \mu&(j_1/g, \ldots, j_b/g)&(m_{j_1}-1, \ldots, m_{j_b}-1)& |X_{\mu,1}|&E_{\mu, 1}&\text{orientable?}\\
 \hline
 6&(1)&(0)&6&\Delta^0&\yes\\
 1+5&(1,5)&(0,0)&1&\bU&\yes\\
 2+4&(1,2)&(0,0)&2&\bU&\yes\\
 1+1+4&(1,4)&(1,0)&1&\bU\times\Delta^1&\yes\\
 3+3&(1)&(1)&3&\Delta^1&\yes\\
 1+2+3&(1,2,3)&(0,0,0)&1&(\bU)^2&\yes\\
 1+1+1+3&(1,3)&(2,0)&1&\bU\widetilde\times\Delta^2\cong \bU\times\Delta^2&\yes\\
 2+2+2&(1)&(2)&2&\Delta^2&\yes\\
 1+1+2+2&(1,2)&(1,1)&1&\bU\widetilde\times(\Delta^1\times\Delta^1)&\no\\
 1+1+1+1+2&(1,2)&(3,0)&1&\bU\widetilde\times\Delta^3\cong \bU\times\Delta^3&\yes\\
 1+1+1+1+1+1&(1)&(5)&1&\Delta^5&\yes\\
\end{array} 
\]

Here the notation $\bU$ denotes the unit circle while $\Delta^p$ denotes the standard p-simplex. The notation $\bU\widetilde\times F$ denotes a twisted bundle over the circle with fibre $F$, which in our case is a simplex or, more generally a polysimplex. (Recall that since $k=1$ there is no cyclic group action on the fibre.)

In the case of the $1+1+4$ partition the gluing map is $W_1^4W_4^{-1}$ which acts as the identity on the fibre so the bundle is trivial as noted in the table. In the case when $\mu=1+1+1+3$ the gluing map is $W_1^3W_3^{-1}$ which acts as a rotation of order 3 on the simplex $\Delta^2$. While this is non-trivial as a simplex bundle it is orientable and homeomorphic to the direct product.

For $\mu=1+1+2+2$ the fibre is a product of two intervals, and the gluing map, $W_1^2W_2^{-1}$ preserves the first factor and flips the second. It follows that the bundle is the product of a M\"obius band with an interval and, in particular is non-orientable. We note that modulo $2$ the vectors in the table reduce to $(1,0)$ and $(1,1)$ which are linearly independent.

When $\mu = 1+1+1+1+2$ the gluing map is again  $W_1^2W_2^{-1}$ acting as the square of a cyclic permutation of the vertices of the fibre which is a $3$-simplex. This should be thought of as the join of two $1$-simplices both of which are flipped by the gluing map, so this is a rotation of $\pi$ about an axis orthogonal to both $1$-simplices. The result is therefore homeomorphic to the direct product.
\end{example}

\section{Proof of Corollary \ref{corollary2}}

The gluing maps defining the bundle $E_{\mu,\omega}$ and the action of the cyclic group $C_{\gcd(m(\mu),k/|\omega|)}$ are products of simplicial maps on the factors of the polysimplex.  It follows that the product of the barycentres of the simplices is preserved by these actions and that the retraction onto this point defines a homotopy equivalence from the quotient $E_{\mu,\omega}/C_{\gcd(m(\mu),k/|\omega|)}$ to the $b-1$-torus over which the bundle $E_{\mu,\omega}$ is defined.

Each component is thus homotopy equivalent to $(\bU)^{b-1}$ as required.  The computation of the number of components furnished by each $\mu$ is as in Section \ref{proof of corollary}.

\newpage
\begin{appendix}
\section{Computations}
Our formula, together with the binomial formula for the Betti numbers of tori, allow us to compute the cohomology of the extended quotients corresponding to the extended affine Weyl groups considered in this paper. From the point of view of cohomology the real and complex tori $(\bU)^{b-1}, (\C^\times)^{b-1}$ are identical, hence the cohomology groups of the extended quotients $\bS_k\q W$ and $T_k\q W$ are the same. Since $K$-theory is a compactly supported theory one must be more careful about the identification of the real and complex tori, although they do, coincidentally, have the same $K$-theory. For this reason we restrict to the real case when considering $K$ theory so that the standard Chern character argument \cite{BC} provides the ranks of the K-theory groups. 

In Tables \ref{b(n,1)} and \ref{b(n,2)} below we compute the Betti numbers for small values of $n$ in the case of $\SL_n(\C)/C_k$ for $k=1,2$.  By duality these are the same as the Betti numbers for $\PSL_n(\C)$ and $\SL_n(\C)/C_{n/2}$ respectively. For $n_0$ a triangle number there is only one partition that can contribute a top dimensional class, that is $n_0=1+2+\ldots+b$ so the dimension increases and the top dimensional Betti number resets to $1$ at each of these. More generally we have the following.

\begin{lemma}
For each $n$ let $b=\lfloor \frac{\sqrt{8n+1}-3}{2}\rfloor$ and let $n_0=1+2+\ldots+b$, which is the largest triangle number less than or equal to $n$. Then $n=n_0+r$ where $0\leq r\leq b$,  the top dimensional cohomology  of $\bS\q W$ appears in dimension $b-1$, and has rank $P_2(r):=\sum\limits_{s=0}^r P(s)P(r-s)$, where $P(s)$ denotes the number of partitions of $s$ and  $P_2(r)$ is the number of partitions of $r$ into parts of two kinds. This provides a generating function for the top dimensional Betti numbers $P_2(r)$
\[
\prod\limits_{s=1}^\infty\frac{1}{(1-x^s)^2}.
\] 
\end{lemma}

\begin{proof}
The top dimensional cohomology is carried by tori corresponding to partitions which maximise the number of distinct parts $b(\mu)$, which increases at each triangle number since these are the minimal numbers which can be partitioned into a given number of distinct parts. So the dimension of the cohomology of the extended quotient is $\lfloor \frac{\sqrt{8n+1}-3}{2}\rfloor$.

Given the partition $n_0=1+2+\ldots+b$ we obtain the partitions for $n+r$ as follows. First partition $r$ as a sum $s+(r-s)$ then choose partitions of $s$ and $r-s$. We append the partition of $s$ to the partition of $n_0$ noting that this does not change the number of distinct parts but increases multiplicities. Now writing the partition of $r-s=x_1+x_2+\ldots+x_d$ in increasing order we add the terms to the last $d$ terms of the partition, which again preserves the number $b$ of distinct parts. Conversely, suppose we are given a partition of $n=n_0+r$ with $b$ distinct parts $j_1< j_2<\ldots<j_b$, which have multiplicities $m_{j_1},\ldots, m_{j_b}$. We express the partition as the union of the partition of $j_1+j_2+\ldots + j_b$ with parts $j_i$ together with its complement in the partition of $n$. Since $j_i\geq i$ it follows that $j_1+j_2+\ldots + j_b\geq n_0$, so the complementary partition is a partition of $s\leq b$, with multiplicities $m_{j_1}-1, \ldots m_{j_b}-1$. The partition $j_1+j_2+\ldots + j_b$ of $n_0+(r-s)$ is obtained from the partition $n_0=1+2+\ldots+b$ by adding $j_i- i$ to the $i$th term, where the non-zero terms $j_i-i$ define a partition of $r-s$ where $j_{i-1}-(i-1)\leq j_i-i$. Hence this produces all partitions of $n$ into exactly $b$ parts.
\end{proof}

Table \ref{K(n,k)} gives the ranks of $K_0$ and $K_1$ for $n$ ranging from $2$ to $20$ and $k$ dividing $n$ in the case of the Lie groups $\SU_n(\C)/C_k$. Note that, as remarked above, when $n$ is square free the answers do not vary with $k$. The Euler characteristic of the extended quotient is the difference of the ranks of $K_0$ and $K_1$, and we note that for $\SU_n(\C)$ (and, by duality, $\PSU_n(\C)$) this is the sum of the divisors of $n$. This follows from considering the partitions of $n$ with a single distinct part and summing the size of these parts.

\begin{table}[ht]
\[\begin{array}{|l|r|r|r|r|r|r|r|r|r|}
\hline
n\hfill&\hfill b_0\hfill&\hfill b_1\hfill&\hfill b_2\hfill&\hfill b_3\hfill&\hfill b_4\hfill&\hfill b_5\hfill&\hfill b_6\hfill&\hfill b_7\hfill&\hfill b_8 \\
\hline
 1 & 1 & \text{} & \text{} & \text{} & \text{} & \text{} & \text{} & \text{} & \text{} \\
 2 & 3 & \text{} & \text{} & \text{} & \text{} & \text{} & \text{} & \text{} & \text{} \\
 3 & 5 & 1 & \text{} & \text{} & \text{} & \text{} & \text{} & \text{} & \text{} \\
 4 & 9 & 2 & \text{} & \text{} & \text{} & \text{} & \text{} & \text{} & \text{} \\
 5 & 11 & 5 & \text{} & \text{} & \text{} & \text{} & \text{} & \text{} & \text{} \\
 6 & 20 & 9 & 1 & \text{} & \text{} & \text{} & \text{} & \text{} & \text{} \\
 7 & 21 & 15 & 2 & \text{} & \text{} & \text{} & \text{} & \text{} & \text{} \\
 8 & 35 & 25 & 5 & \text{} & \text{} & \text{} & \text{} & \text{} & \text{} \\
 9 & 42 & 39 & 10 & \text{} & \text{} & \text{} & \text{} & \text{} & \text{} \\
 10 & 61 & 60 & 18 & 1 & \text{} & \text{} & \text{} & \text{} & \text{} \\
 11 & 66 & 83 & 31 & 2 & \text{} & \text{} & \text{} & \text{} & \text{} \\
 12 & 112 & 132 & 53 & 5 & \text{} & \text{} & \text{} & \text{} & \text{} \\
 13 & 113 & 171 & 82 & 10 & \text{} & \text{} & \text{} & \text{} & \text{} \\
 14 & 168 & 253 & 129 & 20 & \text{} & \text{} & \text{} & \text{} & \text{} \\
 15 & 210 & 346 & 193 & 34 & 1 & \text{} & \text{} & \text{} & \text{} \\
 16 & 279 & 480 & 290 & 60 & 2 & \text{} & \text{} & \text{} & \text{} \\
 17 & 313 & 618 & 415 & 97 & 5 & \text{} & \text{} & \text{} & \text{} \\
 18 & 461 & 882 & 607 & 157 & 10 & \text{} & \text{} & \text{} & \text{} \\
 19 & 508 & 1107 & 841 & 242 & 20 & \text{} & \text{} & \text{} & \text{} \\
 20 & 719 & 1533 & 1192 & 372 & 36 & \text{} & \text{} & \text{} & \text{} \\
 21 & 852 & 1958 & 1627 & 551 & 63 & 1 & \text{} & \text{} & \text{} \\
 22 & 1088 & 2587 & 2248 & 816 & 105 & 2 & \text{} & \text{} & \text{} \\
 23 & 1277 & 3253 & 3006 & 1173 & 172 & 5 & \text{} & \text{} & \text{} \\
 24 & 1756 & 4376 & 4103 & 1685 & 272 & 10 & \text{} & \text{} & \text{} \\
 25 & 2006 & 5400 & 5387 & 2365 & 423 & 20 & \text{} & \text{} & \text{} \\
 26 & 2573 & 7031 & 7212 & 3318 & 642 & 36 & \text{} & \text{} & \text{} \\
 27 & 3106 & 8802 & 9403 & 4563 & 961 & 65 & \text{} & \text{} & \text{} \\
 28 & 3937 & 11304 & 12393 & 6277 & 1414 & 108 & 1 & \text{} & \text{} \\
 29 & 4593 & 13895 & 15942 & 8486 & 2054 & 180 & 2 & \text{} & \text{} \\
 30 & 5958 & 17909 & 20840 & 11480 & 2945 & 287 & 5 & \text{} & \text{} \\
 31 & 6872 & 21787 & 26510 & 15295 & 4175 & 453 & 10 & \text{} & \text{} \\
 32 & 8676 & 27629 & 34226 & 20394 & 5858 & 694 & 20 & \text{} & \text{} \\
 33 & 10305 & 33853 & 43311 & 26834 & 8138 & 1055 & 36 & \text{} & \text{} \\
 34 & 12655 & 42271 & 55286 & 35328 & 11213 & 1566 & 65 & \text{} & \text{} \\
 35 & 15009 & 51480 & 69364 & 45962 & 15313 & 2306 & 110 & \text{} & \text{} \\
 36 & 18664 & 64348 & 88029 & 59864 & 20768 & 3340 & 183 & 1 & \text{} \\
 37 & 21673 & 77496 & 109523 & 77103 & 27944 & 4796 & 295 & 2 & \text{} \\
 38 & 26559 & 95862 & 137729 & 99418 & 37385 & 6796 & 468 & 5 & \text{} \\
 39 & 31447 & 115954 & 170716 & 126960 & 49653 & 9560 & 724 & 10 & \text{} \\
 40 & 38217 & 142322 & 213011 & 162237 & 65632 & 13298 & 1107 & 20 & \text{} \\
 41 & 44623 & 170725 & 262212 & 205495 & 86178 & 18375 & 1660 & 36 & \text{} \\
 42 & 54386 & 209199 & 325553 & 260569 & 112690 & 25161 & 2461 & 65 & \text{} \\
 43 & 63303 & 249804 & 398441 & 327617 & 146468 & 34234 & 3597 & 110 & \text{} \\
 44 & 76379 & 303841 & 491402 & 412339 & 189689 & 46224 & 5203 & 185 & \text{} \\
 45 & 89696 & 363217 & 599369 & 515152 & 244298 & 62058 & 7439 & 298 & 1 \\
 \hline
\end{array}\]
\caption{The Betti numbers of the extended quotient of the maximal torus of $\SL_n(\C)$ by the Weyl group.\label{b(n,1)}}
\end{table}

\begin{table}[ht]
\[\begin{array}{|l|r|r|r|r|r|r|r|r|r|r|}
\hline
n\hfill&\hfill b_0\hfill&\hfill b_1\hfill&\hfill b_2\hfill&\hfill b_3\hfill&\hfill b_4\hfill&\hfill b_5\hfill&\hfill b_6\hfill&\hfill b_7\hfill&\hfill b_8 \hfill&\hfill b_9 \hfill\\
\hline
2 & 3 & \text{} & \text{} & \text{} & \text{} & \text{} & \text{} & \text{} & \text{} & \text{} \\
 4 & 10 & 2 & \text{} & \text{} & \text{} & \text{} & \text{} & \text{} & \text{} & \text{} \\
 6 & 20 & 9 & 1 & \text{} & \text{} & \text{} & \text{} & \text{} & \text{} & \text{} \\
 8 & 40 & 27 & 5 & \text{} & \text{} & \text{} & \text{} & \text{} & \text{} & \text{} \\
 10 & 61 & 60 & 18 & 1 & \text{} & \text{} & \text{} & \text{} & \text{} & \text{} \\
 12 & 122 & 139 & 54 & 5 & \text{} & \text{} & \text{} & \text{} & \text{} & \text{} \\
 14 & 168 & 253 & 129 & 20 & \text{} & \text{} & \text{} & \text{} & \text{} & \text{} \\
 16 & 306 & 505 & 295 & 60 & 2 & \text{} & \text{} & \text{} & \text{} & \text{} \\
 18 & 461 & 882 & 607 & 157 & 10 & \text{} & \text{} & \text{} & \text{} & \text{} \\
 20 & 758 & 1583 & 1210 & 373 & 36 & \text{} & \text{} & \text{} & \text{} & \text{} \\
 22 & 1088 & 2587 & 2248 & 816 & 105 & 2 & \text{} & \text{} & \text{} & \text{} \\
 24 & 1848 & 4504 & 4156 & 1690 & 272 & 10 & \text{} & \text{} & \text{} & \text{} \\
 26 & 2573 & 7031 & 7212 & 3318 & 642 & 36 & \text{} & \text{} & \text{} & \text{} \\
 28 & 4063 & 11527 & 12518 & 6297 & 1414 & 108 & 1 & \text{} & \text{} & \text{} \\
 30 & 5958 & 17909 & 20840 & 11480 & 2945 & 287 & 5 & \text{} & \text{} & \text{} \\
 32 & 8939 & 28109 & 34516 & 20454 & 5860 & 694 & 20 & \text{} & \text{} & \text{} \\
 34 & 12655 & 42271 & 55286 & 35328 & 11213 & 1566 & 65 & \text{} & \text{} & \text{} \\
 36 & 19041 & 65152 & 88616 & 60021 & 20778 & 3340 & 183 & 1 & \text{} & \text{} \\
 38 & 26559 & 95862 & 137729 & 99418 & 37385 & 6796 & 468 & 5 & \text{} & \text{} \\
 40 & 38892 & 143835 & 214203 & 162609 & 65668 & 13298 & 1107 & 20 & \text{} & \text{} \\
 42 & 54386 & 209199 & 325553 & 260569 & 112690 & 25161 & 2461 & 65 & \text{} & \text{} \\
 44 & 77335 & 306262 & 493588 & 413151 & 189794 & 46226 & 5203 & 185 & \text{} & \text{} \\
 46 & 106879 & 438330 & 734529 & 643953 & 313641 & 82762 & 10542 & 476 & 2 & \text{} \\
 48 & 151344 & 633466 & 1092391 & 995271 & 510521 & 144834 & 20594 & 1137 & 10 & \text{} \\
 50 & 206440 & 893139 & 1596122 & 1515435 & 817909 & 248268 & 38983 & 2555 & 36 & \text{} \\
 52 & 286682 & 1268240 & 2329944 & 2290931 & 1294065 & 417826 & 71771 & 5463 & 110 & \text{} \\
 54 & 390133 & 1771783 & 3355161 & 3420111 & 2020028 & 691336 & 128924 & 11196 & 300 & \text{} \\
 56 & 534934 & 2480996 & 4820172 & 5072447 & 3119948 & 1126632 & 226551 & 22131 & 747 & 2 \\
 58 & 719869 & 3424517 & 6843460 & 7441500 & 4763117 & 1809963 & 390279 & 42387 & 1742 & 10 \\
 60 & 979775 & 4745295 & 9701610 & 10856045 & 7205703 & 2870751 & 660346 & 78961 & 3846 & 36 \\

 \hline
\end{array}
\]
\caption{The Betti numbers of the extended quotient of the maximal torus of $\SL_n(\C)/C_2$ by the Weyl group.\label{b(n,2)}}
\end{table}

\newpage
\newgeometry{margin=0.3in}
\begin{landscape}
\begin{table}[ht]
\[
\begin{array}{|c|r|r|r|r|r|r|r|r|r|r|r|r|r|r|r|r|r|r|r|r|}
 \hline
n, k&1&2&3&4&5&6&7&8&9&10&11&12&13&14&15&16&17&18&19&20\\
 \hline
 2 & 
\begin{array}{c}
 3 \\
 0 \\
\end{array}
 & 
\begin{array}{c}
 3 \\
 0 \\
\end{array}
 & \text{} & \text{} & \text{} & \text{} & \text{} & \text{} & \text{} & \text{} & \text{} & \text{} & \text{} & \text{} & \text{} & \text{} & \text{}
& \text{} & \text{} & \text{} \\ \hline
 3 & 
\begin{array}{c}
 5 \\
 1 \\
\end{array}
 & \text{} & 
\begin{array}{c}
 5 \\
 1 \\
\end{array}
 & \text{} & \text{} & \text{} & \text{} & \text{} & \text{} & \text{} & \text{} & \text{} & \text{} & \text{} & \text{} & \text{} & \text{} & \text{}
& \text{} & \text{} \\ \hline
 4 & 
\begin{array}{c}
 9 \\
 2 \\
\end{array}
 & 
\begin{array}{c}
 10 \\
 2 \\
\end{array}
 & \text{} & 
\begin{array}{c}
 9 \\
 2 \\
\end{array}
 & \text{} & \text{} & \text{} & \text{} & \text{} & \text{} & \text{} & \text{} & \text{} & \text{} & \text{} & \text{} & \text{} & \text{} & \text{}
& \text{} \\ \hline
 5 & 
\begin{array}{c}
 11 \\
 5 \\
\end{array}
 & \text{} & \text{} & \text{} & 
\begin{array}{c}
 11 \\
 5 \\
\end{array}
 & \text{} & \text{} & \text{} & \text{} & \text{} & \text{} & \text{} & \text{} & \text{} & \text{} & \text{} & \text{} & \text{} & \text{} & \text{}\\
 \hline
 6 & 
\begin{array}{c}
 21 \\
 9 \\
\end{array}
 & 
\begin{array}{c}
 21 \\
 9 \\
\end{array}
 & 
\begin{array}{c}
 21 \\
 9 \\
\end{array}
 & \text{} & \text{} & 
\begin{array}{c}
 21 \\
 9 \\
\end{array}
 & \text{} & \text{} & \text{} & \text{} & \text{} & \text{} & \text{} & \text{} & \text{} & \text{} & \text{} & \text{} & \text{} & \text{} \\ \hline
 7 & 
\begin{array}{c}
 23 \\
 15 \\
\end{array}
 & \text{} & \text{} & \text{} & \text{} & \text{} & 
\begin{array}{c}
 23 \\
 15 \\
\end{array}
 & \text{} & \text{} & \text{} & \text{} & \text{} & \text{} & \text{} & \text{} & \text{} & \text{} & \text{} & \text{} & \text{} \\ \hline
 8 & 
\begin{array}{c}
 40 \\
 25 \\
\end{array}
 & 
\begin{array}{c}
 45 \\
 27 \\
\end{array}
 & \text{} & 
\begin{array}{c}
 45 \\
 27 \\
\end{array}
 & \text{} & \text{} & \text{} & 
\begin{array}{c}
 40 \\
 25 \\
\end{array}
 & \text{} & \text{} & \text{} & \text{} & \text{} & \text{} & \text{} & \text{} & \text{} & \text{} & \text{} & \text{} \\ \hline
 9 & 
\begin{array}{c}
 52 \\
 39 \\
\end{array}
 & \text{} & 
\begin{array}{c}
 56 \\
 41 \\
\end{array}
 & \text{} & \text{} & \text{} & \text{} & \text{} & 
\begin{array}{c}
 52 \\
 39 \\
\end{array}
 & \text{} & \text{} & \text{} & \text{} & \text{} & \text{} & \text{} & \text{} & \text{} & \text{} & \text{} \\ \hline
 10 & 
\begin{array}{c}
 79 \\
 61 \\
\end{array}
 & 
\begin{array}{c}
 79 \\
 61 \\
\end{array}
 & \text{} & \text{} & 
\begin{array}{c}
 79 \\
 61 \\
\end{array}
 & \text{} & \text{} & \text{} & \text{} & 
\begin{array}{c}
 79 \\
 61 \\
\end{array}
 & \text{} & \text{} & \text{} & \text{} & \text{} & \text{} & \text{} & \text{} & \text{} & \text{} \\ \hline
 11 & 
\begin{array}{c}
 97 \\
 85 \\
\end{array}
 & \text{} & \text{} & \text{} & \text{} & \text{} & \text{} & \text{} & \text{} & \text{} & 
\begin{array}{c}
 97 \\
 85 \\
\end{array}
 & \text{} & \text{} & \text{} & \text{} & \text{} & \text{} & \text{} & \text{} & \text{} \\ \hline
 12 & 
\begin{array}{c}
 165 \\
 137 \\
\end{array}
 & 
\begin{array}{c}
 176 \\
 144 \\
\end{array}
 & 
\begin{array}{c}
 165 \\
 137 \\
\end{array}
 & 
\begin{array}{c}
 165 \\
 137 \\
\end{array}
 & \text{} & 
\begin{array}{c}
 176 \\
 144 \\
\end{array}
 & \text{} & \text{} & \text{} & \text{} & \text{} & 
\begin{array}{c}
 165 \\
 137 \\
\end{array}
 & \text{} & \text{} & \text{} & \text{} & \text{} & \text{} & \text{} & \text{} \\ \hline
 13 & 
\begin{array}{c}
 195 \\
 181 \\
\end{array}
 & \text{} & \text{} & \text{} & \text{} & \text{} & \text{} & \text{} & \text{} & \text{} & \text{} & \text{} & 
\begin{array}{c}
 195 \\
 181 \\
\end{array}
 & \text{} & \text{} & \text{} & \text{} & \text{} & \text{} & \text{} \\ \hline
 14 & 
\begin{array}{c}
 297 \\
 273 \\
\end{array}
 & 
\begin{array}{c}
 297 \\
 273 \\
\end{array}
 & \text{} & \text{} & \text{} & \text{} & 
\begin{array}{c}
 297 \\
 273 \\
\end{array}
 & \text{} & \text{} & \text{} & \text{} & \text{} & \text{} & 
\begin{array}{c}
 297 \\
 273 \\
\end{array}
 & \text{} & \text{} & \text{} & \text{} & \text{} & \text{} \\ \hline
 15 & 
\begin{array}{c}
 404 \\
 380 \\
\end{array}
 & \text{} & 
\begin{array}{c}
 404 \\
 380 \\
\end{array}
 & \text{} & 
\begin{array}{c}
 404 \\
 380 \\
\end{array}
 & \text{} & \text{} & \text{} & \text{} & \text{} & \text{} & \text{} & \text{} & \text{} & 
\begin{array}{c}
 404 \\
 380 \\
\end{array}
 & \text{} & \text{} & \text{} & \text{} & \text{} \\ \hline
 16 & 
\begin{array}{c}
 571 \\
 540 \\
\end{array}
 & 
\begin{array}{c}
 603 \\
 565 \\
\end{array}
 & \text{} & 
\begin{array}{c}
 609 \\
 569 \\
\end{array}
 & \text{} & \text{} & \text{} & 
\begin{array}{c}
 603 \\
 565 \\
\end{array}
 & \text{} & \text{} & \text{} & \text{} & \text{} & \text{} & \text{} & 
\begin{array}{c}
 571 \\
 540 \\
\end{array}
 & \text{} & \text{} & \text{} & \text{} \\ \hline
 17 & 
\begin{array}{c}
 733 \\
 715 \\
\end{array}
 & \text{} & \text{} & \text{} & \text{} & \text{} & \text{} & \text{} & \text{} & \text{} & \text{} & \text{} & \text{} & \text{} & \text{} & \text{}
& 
\begin{array}{c}
 733 \\
 715 \\
\end{array}
 & \text{} & \text{} & \text{} \\ \hline
 18 & 
\begin{array}{c}
 1078 \\
 1039 \\
\end{array}
 & 
\begin{array}{c}
 1078 \\
 1039 \\
\end{array}
 & 
\begin{array}{c}
 1102 \\
 1057 \\
\end{array}
 & \text{} & \text{} & 
\begin{array}{c}
 1102 \\
 1057 \\
\end{array}
 & \text{} & \text{} & 
\begin{array}{c}
 1078 \\
 1039 \\
\end{array}
 & \text{} & \text{} & \text{} & \text{} & \text{} & \text{} & \text{} & \text{} & 
\begin{array}{c}
 1078 \\
 1039 \\
\end{array}
 & \text{} & \text{} \\ \hline
 19 & 
\begin{array}{c}
 1369 \\
 1349 \\
\end{array}
 & \text{} & \text{} & \text{} & \text{} & \text{} & \text{} & \text{} & \text{} & \text{} & \text{} & \text{} & \text{} & \text{} & \text{} & \text{}
& \text{} & \text{} & 
\begin{array}{c}
 1369 \\
 1349 \\
\end{array}
 & \text{} \\ \hline
 20 & 
\begin{array}{c}
 1947 \\
 1905 \\
\end{array}
 & 
\begin{array}{c}
 2004 \\
 1956 \\
\end{array}
 & \text{} & 
\begin{array}{c}
 1947 \\
 1905 \\
\end{array}
 & 
\begin{array}{c}
 1947 \\
 1905 \\
\end{array}
 & \text{} & \text{} & \text{} & \text{} & 
\begin{array}{c}
 2004 \\
 1956 \\
\end{array}
 & \text{} & \text{} & \text{} & \text{} & \text{} & \text{} & \text{} & \text{} & \text{} & 
\begin{array}{c}
 1947 \\
 1905 \\
\end{array}
 \\
 \hline
\end{array}
\]
\caption{The rank of the K-theory  of the extended quotient of the maximal torus of $\SU_n(\C)/C_k$ by the Weyl group.\label{K(n,k)}}
\end{table}
\end{landscape}
\newpage
\restoregeometry

\end{appendix}

\end{document}